\documentclass[11pt,regno]{amsart}
\usepackage{euscript,graphicx,overpic}
\usepackage{epstopdf,rotating}
\newtheorem{theorem}{Theorem}
\newtheorem{lemma}{Lemma}
\newtheorem{corollary}{Corollary}
\newtheorem{proposition}{Proposition}
\theoremstyle{definition}

\newtheorem*{remark*}{Remark}

\newcommand{\eqdef}{\stackrel{\scriptscriptstyle\rm def}{=}}

\DeclareMathOperator{\diam}{diam}
\DeclareMathOperator{\card}{card}
\DeclareMathOperator{\supp}{supp}
\DeclareMathOperator{\dist}{dist}
\DeclareMathOperator{\Dist}{Dist}

\DeclareMathOperator{\Crit}{{\rm Crit}}

\DeclareMathOperator{\Comp}{{\rm Comp}}
\DeclareMathOperator{\ess}{{\rm ess}}

\def\bC{\mathbb{C}}

\def\bR{\mathbb{R}}

\def\cM{\EuScript{M}}

\def\cL{\EuScript{L}}
\def\cO{\EuScript{O}}

\def\e{{\varepsilon}}

\DeclareMathSymbol{\varnothing}{\mathord}{AMSb}{"3F}
\renewcommand{\emptyset}{\varnothing}

\author{K. Gelfert} \address{Instituto de Matem\'atica, UFRJ,
Cidade Universit\'aria - Ilha do Fund\~ao, Rio de Janeiro 21945-909,  Brazil}
\email{gelfert@im.ufrj.br}

\author{F. Przytycki} \address{Instytut Matematyczny Polskiej Akademii Nauk, ul. \'{S}niadeckich 8, 00-956 Warszawa, Poland}
\email{feliksp@impan.gov.pl}

\author{M. Rams} \address{Instytut Matematyczny Polskiej Akademii Nauk, ul. \'{S}niadeckich 8, 00-956 Warszawa, Poland}
\email{rams@impan.gov.pl}

\author{J. Rivera-Letelier} \address{Facultad de Matem{\'a}ticas, Pontificia Universidad Cat{\'o}lica de Chile, Avenida Vicu{\~n}a Mackenna~4860, Santiago, Chile}
\email{riveraletelier@mat.puc.cl}

\begin{document}

\title[Lyapunov spectrum]{Lyapunov spectrum\\ for exceptional rational maps}

\begin{abstract}
We study the dimension spectrum for Lyapunov exponents for
rational maps acting on the Riemann sphere and characterize it by
means of the Legendre-Fenchel transform of the hidden variational
pressure. This pressure is defined by means of the variational
principle with respect to non-atomic invariant probability
measures and is associated to certain $\sigma$-finite conformal
measures. This allows to extend previous results to exceptional
rational maps.
\end{abstract}


\maketitle

\section{Introduction and main results}\label{sec:intro}

We are going to study the Lyapunov exponents of a rational
function $f\colon \overline\bC\to\overline\bC$ acting on the
Riemann sphere, of degree at least~$2$. In particular, continuing
the investigations in~\cite{GelPrzRam:}, we are interested in the
case that the map~$f$ is exceptional. Slightly
modifying~\cite[Section 1.3]{MakSmi:00}, we call~$f$
\emph{exceptional} if there exists a finite, nonempty, and forward
invariant set $\Sigma' \subset J$ such that
\begin{equation}\label{e.except}
f^{-1}(\Sigma') \setminus \Sigma' \subset \Crit \,.
\end{equation}
Here $J=J(f)$ is the Julia set of~$f$ and $\Crit = \Crit(f)$ is
the set of critical points of~$f$. Every such set~$\Sigma'$ has at
most~$4$ points (see Lemma~\ref{lem:at most four}), hence there is
a maximal set with this property, which we denote by~$\Sigma(f)$.
If~$f$ is non-exceptional we put~$\Sigma(f) = \emptyset$. When~$f$
is clear from the context we denote~$\Sigma(f)$ simply
by~$\Sigma$.

\subsection{Main results}

Given $x\in J$, denote by $\underline\chi(x)$ and $\overline\chi(x)$ the
\emph{lower} and \emph{upper Lyapunov exponent} at $x$,
respectively.
If both values coincide then we call the common
value the \emph{Lyapunov exponent} at $x$ and denote it by~$\chi(x)$.
Similarly, for a $f$-invariant probability measure~$\mu$ we denote by~$\chi(\mu) \eqdef \int \log\, \lvert f'\rvert\, d \mu$ its \emph{Lyapunov exponent}.
Let $\cM$ be the set of all $f$-invariant Borel probability measures supported on $J$ and  $\widetilde\cM\subset\cM$ be the one of all non-atomic ones. Let $\widetilde\cM_E$ and $\cM_E$  be the sets of ergodic measures
contained in $\widetilde\cM$ and $\cM$, respectively. Let
\[
\alpha^- \eqdef \inf_{\mu\in \cM_E}\chi(\mu),
\quad\alpha^+ \eqdef \sup_{\mu\in\cM_E}\chi(\mu),
\quad\widetilde\alpha^+ \eqdef \sup_{\mu\in\widetilde\cM_E}\chi(\mu),
\]
(see Corollary~\ref{def} for equivalent definitions of $\widetilde{\alpha}^+$).

For given numbers $0\le\alpha\le\beta$ we consider the level sets
\[
\cL(\alpha,\beta)\eqdef \{x\in J\colon \underline\chi(x)=\alpha,\overline\chi(x)=\beta\}.
\]
We denote by $\cL(\alpha)\eqdef\cL(\alpha,\alpha)$ the set of \emph{Lyapunov regular points} with exponent $\alpha$. We will describe the complexity of such level sets in terms of their Hausdorff dimension $\dim_{\rm H}$.
To do so, given a parameter $t\in\bR$ let us consider the potential $\varphi_t\eqdef -t\log\,\lvert f'\rvert$ and the pressure function
\begin{equation}\label{def:vp}
   P(\varphi_t) \eqdef
   \sup_{\mu\in \cM}\left( h_\mu(f) +\int_J\varphi_t\,d\mu\right).
\end{equation}
Notice that if $\Crit$ is nonempty the potential $\varphi_t$ is unbounded and~$P(\varphi_t)$ does not coincide with the the classical topological pressure for $t > 0$ (see~\cite{Wal:81}).
We define the {\it hidden variational pressure}
\begin{equation}\label{def:hidden}
   \widetilde P(\varphi_t)\eqdef \sup_{\mu\in\widetilde\cM}
   \left(h_\mu(f)+\int_J\varphi_t\,d\mu \right)
\end{equation}
(following the terminology in~\cite{PrzRivSmi:04}).
After Makarov and Smirnov~\cite[Theorem B]{MakSmi:00}, the pressure function $t\mapsto P(\varphi_t)$ fails to be real analytic on the interval $(-\infty,0)$  if and only if $f$ is
exceptional and
\[
 \chi_{\sup}\eqdef
 \sup_{\mu \in \cM} \chi(\mu)  >
 \sup_{\mu \in \widetilde{\cM}} \chi(\mu).
\]
Moreover, by~\cite[Theorem A]{MakSmi:00} the function $t\mapsto \widetilde P(\varphi_t)$  is real analytic on the interval $(-\infty,0)$ and
\begin{equation}\label{e.MakaSmiri}
P(\varphi_t) = \max \{ \widetilde P(\varphi_t), -t\,\chi_{\sup} \}.
\end{equation}
For any $\alpha> 0$ let
\begin{equation}\label{e.desFtilde}
\widetilde F(\alpha)\eqdef
\frac{1}{\alpha}\inf_{t\in\bR}\left( \widetilde P(\varphi_t)+t\,\alpha\right)
\quad\text{ and }\quad
\widetilde F(0)\eqdef \lim_{\alpha\to 0+}\widetilde F(\alpha).
\end{equation}

Our main result is the following theorem.

\begin{theorem}\label{main1}
   Let $f$ be a rational function of degree at least~$2$. For any numbers $\alpha$, $\beta$ with $0\le\alpha\le \beta \le \widetilde{\alpha}^+$ we have
\[
   \min\{\widetilde F(\alpha),\widetilde F(\beta)\}
   \le \dim_{\rm H} \cL(\alpha, \beta)\le
   \max_{\alpha\leq q \leq \beta}\widetilde F(q)\,.
\]
In particular, for any $\alpha\in [\alpha^-,\widetilde\alpha^+] \setminus
\{0\}$ we have
\[
   \dim_{\rm H} \cL(\alpha) = \widetilde F(\alpha) \,.
\]
For $\alpha=0$ we have
\[
\dim_{\rm H} \cL(0) \geq \widetilde F(0)\,.
\]
Moreover,
\[
\left\{ x\in J\colon -\infty<\chi(x)<\alpha^-\right\}
 = \left\{ x\in J\setminus\Sigma \colon \overline{\chi}(x)>\widetilde \alpha^+\right\}
 = \emptyset
\]
and
\[
\dim_{\rm H}\left\{ x\in J\colon
0<\overline{\chi}(x)<\alpha^-\right\} = 0.
\]
\end{theorem}

The result of the above theorem has been shown in~\cite{GelPrzRam:} in the particular case that $f$ is non-exceptional.

To prove our main result, in this paper we will create new technical tools in order to deal with exceptional rational maps and then show how these tools can be applied to adapt the original proofs in~\cite{GelPrzRam:}.
The paper is organized as follows.
In Section~\ref{sec:2} we collect some known results about exceptional maps that will be used in the rest of the paper.
In Section~\ref{sec:hidden} we will introduce the concept of hidden pressure using backward branches of $f$, analogously to the tree pressure from~\cite{PrzRivSmi:04}.
In the case of exceptional rational maps we not always have at hand a finite conformal measure with dense support, see Proposition~\ref{p:known conformal}.
For that reason, in Section~\ref{sec:bal} we introduce $\sigma$-finite conformal measures that are associated to the hidden pressure.
Finally, in Section~\ref{sec:final} we apply these tools to prove Theorem~\ref{main1}.
In Section~\ref{sec:final-1} we provide a lower bound for dimension using the fact that for any rational map we can find an increasing family of uniformly expanding Cantor repellers contained in $J$ using a construction of bridges that has been established in~\cite{GelPrzRam:} and applies to the setting of this paper without changes.
In Section~\ref{sec:final-2} we provide an upper bound for dimension applying Frostman's Lemma to an appropriate $\sigma$-conformal measure at a conical point. Finally, in Section~\ref{sec:final-3}, we show the existence of periodic orbits  in $J \setminus \Sigma$ with exponent as large as possible.

We give an alternative proof of this result in
Appendix~\ref{s:alternative} via a variant of Bowen's periodic
specification property, \cite{Bow:71}.

\section{Exceptional maps and phase transitions}
\label{sec:2}
For a critical point~$c \in \Crit$ we will denote by~$\deg_f(c)$
the local degree of~$f$ at~$z = c$. The following result has been
proved by the same computation first in~\cite[Lemma 2]{DouHub:93}.

\begin{lemma}\label{lem:at most four}
    If~$\Sigma'$ is a finite subset of~$\overline{\bC}$ such that $f^{-1}(\Sigma') \setminus \Sigma' \subset \Crit$, then $\card\Sigma'\le 4$. If~$f$ is a polynomial then $\card (\Sigma' \setminus \{\infty\}) \le 2$.
\end{lemma}

\begin{proof}
Using that~$f$ has $2 \deg(f) - 2$ critical points counted with multiplicity, by~\eqref{e.except} we have
\begin{multline*}
\deg(f) \card \Sigma'
=
\sum_{x \in f^{-1}(\Sigma')} \deg_f(x)
=
\card f^{-1}(\Sigma') + \sum_{x \in f^{-1}(\Sigma')} (\deg_f(x) - 1)
\\ \le
\card \Sigma' + \card \Crit + 2(\deg(f) - 1)
\le
\card \Sigma' + 4(\deg(f) - 1),
\end{multline*}
so~$\card \Sigma' \le 4$.
If $f$ is a polynomial, then it has at most~$\deg(f) - 1$ finite critical points counted with multiplicity, so in this case $\card(\Sigma' \setminus \{ \infty \} ) \le 2$.
\end{proof}

The following is an example of a one parameter family of exceptional rational maps such that for some parameters the exceptional set contains a critical point: for~$\lambda \in \bC$ put
\[
    f_\lambda(z)=(\lambda z^d-\lambda z^{d-1}+1)^{-1}.
\]
The point~$z = 0$ is critical of multiplicity~$d - 1$, the point $1=f_\lambda(0)$ is fixed of multiplier~$- \lambda$, and the point $z=\infty$ is critical of maximal multiplicity and the only preimage of $z = 0$.
Thus, when~$z = 1$ belongs to the Julia set we have~$\{ 0, 1 \} \subset \Sigma$.
By choosing suitable~$\lambda$, the fixed point~$z = 1$ could be repelling, Cremer, etc.

If~$f$ is exceptional, then the set $\Sigma$ contains at least one
periodic point. Observe that it hence must consist of a finite
number of periodic points plus possibly some of their preimages.
We write $\Sigma = \Sigma_0\cup\Sigma_+$, where $\Sigma_0$ denotes
the subset of all neutral periodic points in $\Sigma$ plus its
pre-images and where $\Sigma_+$ denotes the subset of all
expanding periodic points in $\Sigma$ plus its
pre-images. We refer to~\cite{MakSmi:96} for further details on
exceptional maps and numerous examples.

We will say that~$f$ has a \emph{phase transition in the negative spectrum} if the function~$t \mapsto P(\varphi_t)$ fails to be real analytic on~$(-\infty, 0)$.
In this case we put
\[
   t_- \eqdef \sup \left\{ t < 0 : P(\varphi_t) = - t \chi_{\sup} \right\}.
\]
We have~$t_- < 0$ and, since the function $t \mapsto P(\varphi_t)$ is convex, for each~$t \in (-\infty, t_-)$ we have $P(\varphi_t) = - t \chi_{\sup}$.

In the following proposition we gather several results in~\cite{MakSmi:00, PrzRivSmi:04}. A measurable subset~$A$ of~$\overline{\bC}$ is said to be \emph{special} if  $f\colon A\to f(A)$ is injective.
Given a function $\psi\colon\overline\bC\to\bR$, a Borel probability measure $\nu$ on $J$ is said to be \emph{$e^{\psi}$-conformal outside} $Z\subset J$ if
for every special set $A\subset J\setminus Z$ we have
\[
\nu (f(A))=\int_A  e^{\psi(x)}\, d\nu(x) .
\]
If $Z=\emptyset$ we simply say that~$\nu$ is \emph{$e^\psi$\nobreakdash-conformal}.

\begin{proposition}\label{p:known conformal}
   Let~$f$ be a rational map of degree at least~$2$ and let~$t \in \mathbb{R}$. Then we have the following properties.
\begin{enumerate}
\item[1.] Suppose that~$f$ does not have a phase transition in the
negative spectrum, or that~$f$ has a phase transition in the
negative spectrum and~$t > t_-$. Then~$\widetilde{P}(\varphi_t) =
P(\varphi_t)$ and there is a finite $e^{P(\varphi_t) -
\varphi_t}$-conformal measure whose support is equal to~$J$.
\item[2.] Suppose that~$f$ has a phase transition in the negative
spectrum and that~$t \le t_-$.
Then~$f$ is exceptional, there is an expanding periodic point $p \in \Sigma$ such that $P(\varphi_t) = -t \chi(p)$ and for every neighborhood~$V$ of~$p$ and every measure~$\nu$ that is
$e^{P(\varphi_t) - \varphi_t}$-conformal measure outside~$\Crit$
we have
$$ \nu(V \setminus \{ p \}) = + \infty. $$
\end{enumerate}
\end{proposition}

\begin{proof}
   The equality~$\widetilde{P}(\varphi_t) = P(\varphi_t)$ in part~$1$ follows from the definition of~$t_-$. The existence of the conformal measure in part~$1$ follows
from~\cite[Lemma~$3.5$]{MakSmi:00} if~$t < 0$ and from~\cite[Theorem~A]{PrzRivSmi:04} if~$t \ge 0$.

The fact that~$f$ is exceptional and that there is an expanding periodic point~$p \in \Sigma$ such that~$P(\varphi_t) = - t \chi(p)$ in part~$2$ is given by~\cite[Theorem~B]{MakSmi:00}.
To complete the proof of part~$2$, let~$\nu$ be a $e^{P(\varphi_t) - \varphi_t}$-conformal measure~$\Crit$.
Since~$f$ is topologically exact on~$J$, it follows that the support of~$\nu$ is equal to~$J$.
Let~$n \ge 1$ be the period of~$p$ and let~$r > 0$ be
sufficiently small so that~$B(p, r) \setminus \{ p \} \subset V
\setminus (\Sigma \cup \Crit)$ and so that the inverse
branch~$\phi$ of~$f^n$ fixing~$p$ is defined on~$B(p, r)$ and
satisfies $\phi(\overline{B(p, r)}) \subset B(p, r)$. Then there
is a distortion constant~$C > 0$ such that, if we put~$U \eqdef
B(p, r) \setminus \phi(B(p, r))$, then for each integer~$m \ge 1$
we have by the conformality of~$\nu$
\[
   \nu(\phi^n(U))
   \ge
   C^{-1} \nu(U) e^{- nm (P(\varphi_t) - \varphi_t(p))}
   = C^{-1} \nu(U).
\]
Thus
\[
   \nu(V \setminus \{ p \}) \ge \nu(B(p, r) \setminus \{ p \} )
   = \sum_{m = 1} \nu(\phi^m(U))
   = + \infty
\]
proving the proposition.
\end{proof}

\section{Hidden tree pressure}\label{sec:hidden}

The goal of this section is to prove equivalence of three pressure
functions: the hidden variational pressure defined
in~\eqref{def:hidden} as well as the hidden hyperbolic pressure
and the hidden tree pressure defined in~\eqref{def.hiddhyp}
and~\eqref{def.tree} below.

Given $t\in\bR$, the \emph{hidden hyperbolic pressure} is defined as
\begin{equation}\label{def.hiddhyp}
   \widetilde{P}_{\rm{hyp}}(\varphi_t) \eqdef \sup
   P_{f|X}(\varphi_t),
\end{equation}
where the supremum is taken over all compact $f$-invariant (i.e.
$f(X)\subset X$) isolated expanding subsets of $J\setminus\Sigma$.
We call such a set \emph{uniformly expanding repeller}. Here
\emph{isolated} means that there exists a neighborhood $U$ of $X$
such that $f^n(x)\in U$ for all $n\ge 0$ implies $x\in X$.

\begin{proposition}\label{p.primeiro}
    $\displaystyle \widetilde{P}(\varphi_t)= \widetilde{P}_{\rm{hyp}}(\varphi_t)$ for every $t\in\bR$.
\end{proposition}

\begin{proof}
The inequality $\widetilde{P}(\varphi_t)\ge
\widetilde{P}_{\rm{hyp}}(\varphi_t)$ follows from the variational
principle. On the other hand~\cite[Theorem 11.6.1]{PrzUrb:}
implies that for any $\mu\in\widetilde \cM$ we have $\widetilde{P}_{\rm{hyp}}(\varphi_t)\ge h_\mu(f)+\int_J\varphi_t\,d\mu$  and hence
$\widetilde{P}_{\rm{hyp}}(\varphi_t)\ge \widetilde{P}(\varphi_t)$.
\end{proof}

Before defining the hidden tree pressure, let us recall some
concepts from~\cite{Prz:99},~\cite{PrzRivSmi:04}, and~\cite[Chapter
12.5]{PrzUrb:}. Given $z\in\overline\bC$ and $t\in\bR$, we
consider the \emph{tree pressure} of $\varphi_t$ at $z$ defined by
    \[
        P_{\rm tree}(z,\varphi_t)\eqdef
        \limsup_{n\to\infty}\frac{1}{n}\log \sum_{x\in f^{-n}(z)}\lvert(f^n)'(x)\rvert^{-t}.
    \]
A point $z\in\overline\bC$ is said to be \emph{safe} if
    \[
        z\notin\bigcup_{n=1}^\infty f^n(\Crit)
        \quad\text{ and }\quad
        \lim_{n\to\infty}\frac{1}{n}\log \dist(z,f^n(\Crit))=0,
    \]
where $\dist$ denotes the spherical distance. A point
$z\in\overline\bC$ is said to be \emph{expanding} if there exist
numbers $\Delta>0$ and $\lambda > 1$ such that for all sufficiently large $n$ the map $f^n$ is univalent on $f^{-n}_z(B(f^n(z),\Delta))$ and satisfies
$\lvert (f^n)'(z)\rvert\ge \lambda^n$. Here, for a subset~$U$
of~$\overline{\bC}$ and $z \in U$ we denote by $f^{-n}_z(U)$ the
connected component of~$f^{-n}(U)$
containing~$z$.

We point out that every point in~$\overline{\bC}$ outside a set of
Hausdorff dimension zero is safe and that for each safe point~$z
\in \overline{\bC}$ we have $P_{\rm tree}(z,\varphi_t) =
P(\varphi_t)$, see~\cite{Prz:99,PrzRivSmi:04} and compare
with~\cite[Theorem 3.4]{Prz:99}.

\begin{lemma}\label{lem:exisafe}
   There exists an expanding safe point in $J \setminus \Sigma$.
\end{lemma}

\begin{proof}
    Notice that
    \[
       \{x\in\overline\bC \text{ not safe }\} \subset
     \left(\bigcup_{n=1}^\infty f^n(\Crit)\right)
     \cup
     \left(\bigcup_{\beta\in(0,1)}\bigcap_{n=1}^\infty\bigcup_{k=n}^\infty B(f^k(\Crit),\beta^k)
       \right).
    \]
Since $\Crit$ is finite and $\sum_n\beta^{nt}<\infty$ for any
$\beta\in(0,1)$ and $t>0$, this inclusion implies that the set of
points that fail to be safe has zero Hausdorff dimension. Thus,
the existence of an expanding safe point outside $\overline V$
follows from the existence of uniformly expanding Cantor repellers
outside $\overline V$, for example as derived in~\cite[Lemma
4]{GelPrzRam:}. Note that such repellers always have a positive
Hausdorff dimension by Bowen's formula (see, for example, \cite[Chapter
9.1]{PrzUrb:}).
\end{proof}

Let us now define the hidden tree pressure that is an analogue of the tree pressure,
obtained by considering a restricted tree of preimages.
Given a subset~$V$ of~$J$ and~$z \in J \setminus V$ which is not in the forward orbit of a critical point, we define
\begin{equation}\label{def.tree.1}
    P_n(z,\varphi_t,V)\eqdef
    \frac{1}{n}\log\sum_{x\in f^{-n}(z)\cap J\setminus V}\lvert (f^n)'(x)\rvert^{-t}
\end{equation}
and we consider the \emph{hidden tree pressure} of $\varphi_t$ at $z$ defined by
\begin{equation}\label{def.tree}
   P_{\rm tree}(z,\varphi_t,V) \eqdef \limsup_{n\to\infty} P_n(z,\varphi_t,V).
\end{equation}
Usually the point~$z$ will be expanding safe and the set~$V$ will be a
neighborhood of~$\Sigma$.
Note that after Lemma~\ref{lem:exisafe} there are such~$z$ and~$V$.

\begin{lemma}\label{l.correct}
   If $t \le 0$, $V$ is a sufficiently small neighborhood of~$\Sigma$, and $z \in J \setminus V$ is expanding safe, then the pressure $P_{\rm tree}(z,\varphi_t,V)$ does not depend on~$V$.
\end{lemma}

To prove the above lemma we need the following technical lemma.

\begin{lemma}\label{lem:caf1}
   For an arbitrary neighborhood~$V$ of~$\Sigma$ and an arbitrary number $\varepsilon>0$ there exists a number  $\delta>0$ and positive integers $N\le M$ such that for every point $x\in J\setminus V$ there exist numbers $0\le i$, $j\le M$ and a point $z\in f^{-j}(f^i(\{x\}))$ such that the set $A\eqdef f^{-N}(\{z\})$ is $\varepsilon$-dense in $J$ and satisfies
    \[
        \dist\Big(\bigcup_{s=0}^{N+j-1} f^s(A),\Crit\Big)\ge \delta.
    \]
\end{lemma}

\begin{proof}
By the locally eventually onto property of~$f$ on~$J$ there is an integer~$N \ge 1$ such that for each~$z \in J$ the set~$f^{-N}(z)$ is $\varepsilon$-dense in~$J$. We put
   \[
   C(N)\eqdef\bigcup_{s=1}^{N} f^s(\Crit).
   \]
For each integer $M\ge0$ let $\Phi_M\colon J\to \bR$ be defined by
\begin{multline*}
   \Phi_M(x)\eqdef \max\Big\{\min\big\{\dist(y,C(N)),
                   \dist\big(\bigcup_{s=0}^{j-1}f^s(y),\Crit\big)\big\}\colon\\
       0\le i,j\le M, \,y\in f^{-j}(f^i(x))\Big\}.
\end{multline*}
We will show that for each~$x \in J \setminus \Sigma$ there is an
integer $M(x) \ge 0$ such $\Phi_{M(x)}(x) > 0$. Since for each $M
\ge 0$ the function~$\Phi_M$ is continuous and for each~$x \in J$
the sequence $\{ \Phi_M(x) \}_{M = 0}^{\infty}$ is non-decreasing,
it follows that there is a number $M\ge N$ so that $\Phi_M$ is
strictly positive on $J\setminus \Sigma$. This will imply the desired assertion with
   \[
       \delta\eqdef \inf\{\Phi_M(x)\colon x\in J\setminus V\}\cdot
               (\sup\,\lvert f'(x)\rvert)^{-N}.
   \]
We distinguish three cases:\\
\textbf{1)} If $x\in J\setminus \Sigma$ is not in the forward orbit of a critical point then $\Phi_0(x)>0$.\\
\textbf{2)} If $x\in J\setminus \Sigma$ is in the forward orbit of a citical point that is not pre-periodic then there exists a number $i=i(x)\in\{0,\ldots,\card C(N)\}$ such that $f^i(x)$ is disjoint from $C(N)$. Hence, we obtain that $\Phi_{\card C(N)}(x)>0$.\\
\textbf{3)} If $x\in J\setminus \Sigma$ is in the forward orbit of a pre-periodic critical point then, there is~$i$ and an infinite backward trajectory starting at~$f^i(x)$ that is disjoint from $\Crit$ and in particular this backward trajectory is longer than $\card C(N)$. Hence, we can choose numbers $i=i(x)$, $j=j(x)\ge 0$ and a point $y=y(x)\in f^{-j}(f^i(x))$ such that $y$ is not in the forward orbit of a critical point and such that for each $s\in\{0,\ldots,j-1\}$ we have $f^s(y)\notin\Crit$. In particular, we have $y\notin C(N)$. Thus, if we put
   \[
       M(x)\eqdef \max\{i(x),j(x)\},
   \]
then $\Phi_{M(x)}(x)>0$.
\end{proof}

\begin{proof}[Proof of Lemma~\ref{l.correct}]
Let $V_1$, $V_2$ be two neighborhoods of $\Sigma$. Without loss of
generality we can assume that $V_1\subset V_2$. By Lemma
\ref{lem:caf1}, every backward branch of $f^{-n}$ starting at $z$
and ending at some point $x_1\in
V_2\setminus V_1$ can be modified to end at some $x_2\notin V_2$.
The modification involves only removing at most $M$ last steps,
that decreases $\lvert(f^n)'(x)\rvert^{-t}$ at most by a constant
factor because $t \le 0$,
and replacing them by at most $M+N$ steps, which stay in a
uniformly bounded from below distance from critical points. Hence
we conclude that $P_n(z,\varphi_t,V_1)$ and $P_n(z,\varphi_t,V_2)$
differ at most by $\cO(n^{-1})$. This proves the lemma.
\end{proof}

We denote by $\Dist g|_Z\eqdef \sup_{x,y\in Z}\lvert g'(x)\rvert/\lvert g'(y)\rvert$ the maximal distortion of a map $g$ on a set $Z$. We establish one preliminary approximation result.

\begin{proposition}\label{p.schuh}
Given~$t \le 0$, a sufficiently small neighborhood $V$ of $\Sigma$ and an expanding safe point $z\in J\setminus \Sigma$, for every $\varepsilon>0$ there exists a uniformly expanding repeller $X\subset J\setminus \Sigma$ such that
   \[
       P_{f|X}(\varphi_t)\ge P_{\rm tree}(z,\varphi_t,V)-\varepsilon.
   \]
\end{proposition}

\begin{proof}
We start by recalling the construction used in~\cite[Proposition~2.1]{PrzRivSmi:04} to prove an analogous statement for~$t \ge 0$ and then we modify it using Lemma~\ref{lem:caf1} to prove the proposition.

As~$z$ is expanding safe, there exist $\Delta>0$, $C_0 > 0$ and $\lambda>1$ so that for all~$\ell \ge 1$ the map $f^\ell$ is univalent on $V_\ell\eqdef f^{-\ell}_z(B(f^\ell(z),\Delta))$ and $\lvert(f^\ell)'(z) \rvert \ge C_0 \lambda^\ell$.
Hence, in particular, the distortion $\Dist f^\ell|_{V_\ell}$ is bounded from above uniformly in~$\ell$ by some number~$C_1 > 1$.
Given $r<\Delta/2$, let $\ell=\ell(r)$ be the smallest integer satisfying $\lvert (f^\ell)'(z)\rvert\ge C_1\Delta/r$. Hence, with the above, we have $f^{-\ell}(B(f^\ell(z),\Delta))\subset B(z,r)$ and $\ell\le C''-C'\log r$, where $C'=1/\log\lambda$ and $C''=(\log\lambda+\log C_0^{-1}C_1\Delta)C'$.
Let $m\ge1$ be such that $f^m(B(y,\Delta/2))=J$ for any $y\in J$.

Let us choose positive constants $\alpha$, $\kappa$ and $n\ge m$ large enough so that $\kappa n^{-\alpha}<\Delta/2$ and that for every $j=1$, $\ldots$, $2n$ for every point $z_j\in f^{-j}(z)$ on the component $f^{-j}_{z_j}(B(z,\kappa n^{-\alpha}))$ the map $\displaystyle f^j$ is univalent and satisfies
\begin{equation}\label{e.mimo}
   f^{-j}_{z_j}(B(z,\kappa n^{-\alpha}))\subset B(z_j,\Delta/2).
\end{equation}
Note that with this choice we have for large~$n$,
$$ \ell\eqdef\ell(\kappa n^{-\alpha})\le C''-C'\log\kappa+\alpha\, C'\log n
\ll
n - m. $$
As $m\le n$ and $f^m(B(f^\ell(z),\Delta)))$ covers $J$, we can conclude that for every preimage $z_n\in f^{-n}(z)$ there exists a component~$W_{z_n}$ of $f^{-m}\big(f^{-n}_{z_n}(B(z,\kappa n^{-\alpha}))\big)$ contained in~$B\big(f^\ell(z),\Delta\big)$.
The map $f^{m+n}$, and hence~$f^{m+n+\ell}$, is univalent on~$W_{z_n}$.
Thus, the map
\begin{equation}\label{e.defF}
   F\eqdef f^{m+n+\ell}
   \colon \bigcup_{z_n \in f^{-n}(z)}
W_{z_n} \to
   B\big( f^\ell(z),\Delta\big)
\end{equation}
\begin{figure}[h]
\begin{minipage}[c]{\linewidth}
\centering
\begin{overpic}[scale=.4
 ]{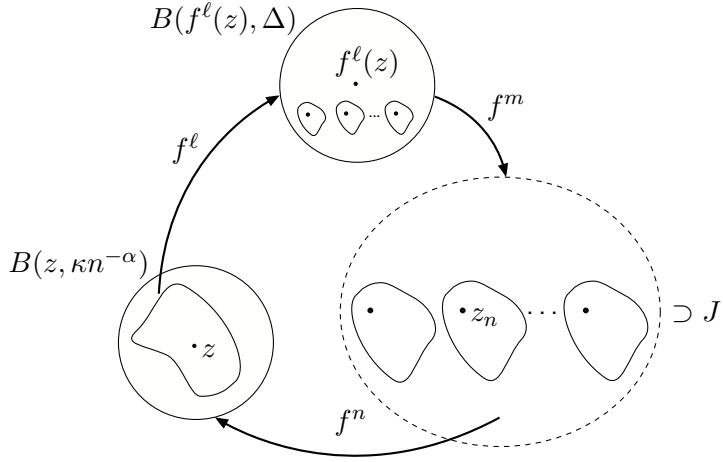}
      \put(6,79){$B(f^\ell(z),\Delta)$}
  \put(40,71){$f^\ell(z)$}
   \put(15.5,18.5){$z$}
   \put(102,25){$\supset J$}
   \put(-20,34){$B(z,\kappa n^{-\alpha})$}
   \put(65,25){$z_n$}
   \put(10,55){$f^\ell$}
   \put(40,5){$f^n$}
   \put(68,63){$f^m$}
 \end{overpic}
\end{minipage}
\caption{Construction of the uniformly expanding repeller $Z$}
\end{figure}
\noindent
has no critical points, and $Z\eqdef \bigcap_{k=1}^\infty F^{-k}(B(f^\ell(z),\Delta))$ is a uniformly expanding repeller with respect to~$F$.

Let us now slightly modify the construction of $Z$ by~\eqref{e.defF} and ignore all those backward branches $f^{-(m+n+\ell)}$ that correspond to a point $z_n\in V$.
Given $\kappa = \Delta/2$, let us consider the positive integers $N\le M$ and the number $\delta>0$ provided by Lemma~\ref{lem:caf1}.
Then, by Lemma~\ref{lem:caf1}, for each point $z_n\in f^{-n}(z) \cap V$ there exist numbers $j(z_n)$, $i(z_n)\le M$, a point $z_n^\ast \in f^{-(N+j(z_n))}(f^{i(z_n)}(z_n))$ in $B(f^\ell(z),\Delta/2)$.
Any such branch stays $\delta$-far from~$\Crit$.
Note that the distortion of~$f^{n + N + j(z_n) - i(z_n) + \ell}$ on
$$ W_{z_n}^\ast \eqdef f_{z_n^\ast}^{-(n + N + j(z_n) - i(z_n))}(B(z,\kappa n^{-\alpha})) $$
is bounded by a constant~$D > 1$ independent of~$z_n$ and~$n$.
Given an integer~$k\in\{n + N - M, \ldots,n + N + M\}$, put
$$ P_k \eqdef \{ z_n \in f^{-n}(z) \setminus V \colon n + N + j(z_n) - i(z_n) = k \}. $$
Note that for distinct~$z_n$ and~$z_n'$ in~$P_k$ the sets~$W_{z_n}^\ast$ and~$W_{z_n'}^\ast$ are disjoint.
Setting
\[
   F_k^\ast \eqdef
   f^{k + \ell} \colon \bigcup_{z_n \in P_k} W_{z_n}^\ast \to B(f^\ell(z), \Delta),
\]
the sets
\[
   Z_k^\ast \eqdef
   \bigcap_{j = 1}^{\infty} (F_k^\ast)^{-j} \left( B(f^{\ell}(z), \Delta)\right)
   \text{ and }
   X_k^* \eqdef
   \bigcup_{j = 0}^{k + \ell - 1} f^j(Z_k^\ast)
\]
are uniformly expanding repellers for~$F_k^\ast$ and~$f$, respectively.
Both of these sets are disjoint from~$\Sigma$ by construction.
On the other hand there, letting
\[
   L \eqdef \min \left\{ 1, \inf_{J \setminus B(\Crit, \delta)} \lvert f'\rvert \right\}
   \quad\text{ and }\quad
   \widetilde{L} \eqdef \max \left\{ 1, \sup_J \,\lvert f'\rvert \right\}
\]
we have
\[\begin{split}
   & P_{f|X_k^\ast}(- t \log \,\lvert f'\rvert)
   \ge
       \frac{1}{k + \ell} P_{F_k^\ast | Z_k^\ast}\big(- t \log |F_k^\ast\rvert\big)\\
   &\ge
   \frac{1}{n + N + M + \ell}
   \log \left( D^{t} \sum_{z_n \in P_k} \lvert (f^{k + \ell})'(z_n^\ast)\rvert^{-t} \right)
   \\
   &\ge
   \frac{1}{n + N + M + \ell} \log \left( D^{t} (C_1^{-1} C_0 \lambda^{\ell})^{-t}
   L^{-t(N + M)} \widetilde{L}^{tM} \sum_{z_n \in P_k} \lvert(f^n)'(z_n)\rvert ^{-t}       \right).
\end{split}\]
Since $\bigcup_{k = n + N - M}^{n + N + M} P_k = f^{-n}(z) \setminus V$, there is~$k$ such that
\[
   \sum_{z_n \in P_k} |(f^n)'(z_n)|^{-t}
   \ge
   \frac{1}{2M + 1} \sum_{z_n \in f^{-n}(z) \setminus V} \lvert (f^n)'(z_n)\rvert^{-t}.
\]
Hence, if we put $\widetilde{D} \eqdef D^{-1}C_1^{-1} C_0\, \lambda^\ell L^{N + M} \widetilde{L}^{- M}$, then
\begin{multline*}
   P_{f|X_k^\ast}(- t \log \,\lvert f'\rvert)
   \\ \ge
   \frac{1}{n + N + M + \ell} \log \left( \widetilde{D}^{-t} \frac{1}{2M + 1}
   \sum_{z_n \in f^{-n}(z) \setminus V} \lvert(f^n)'(z_n)\rvert^{-t} \right).
\end{multline*}
Since~$N$, $M$, $\widetilde{D}$ are independent of~$n$ and $\ell \le C''-C'\log\kappa+\alpha \, C'\log n$, we obtain the desired assertion by taking a sufficiently large~$n$.
\end{proof}

We are now ready to prove one further equivalence.
\begin{proposition} \label{sngk}
Given a sufficiently small neighborhood~$V$ of $\Sigma$ and an
expanding safe point $z\in J\setminus V$, for every $t \le
0$ we have
\[
   P_{\rm tree}(z,\varphi_t, V)
=
\widetilde{P}_{\rm{hyp}}(\varphi_t)
\]
\end{proposition}

\begin{proof}
By Proposition~\ref{p.schuh},  we have $\widetilde P_{\rm hyp}(\varphi_t)\ge P_{\rm tree}(z,\varphi_t,V)$.

In view of Lemma~\ref{l.correct}, to prove the inequality $\widetilde P_{\rm hyp}(\varphi_t)\le P_{\rm tree}(z,\varphi_t,V)$ it is enough to show that for each  expanding  repeller~$X$ that does not intersect~$\Sigma$, there is a neighborhood~$V_0$ of~$\Sigma$ such that $P_{f|X}(\varphi_t) \le P_{\rm tree}(z,\varphi_t,V_0)$.
Notice that for every $x\in X$ and every neighborhood~$V_0$ of~$\Sigma$ disjoint from~$X$ we have,
\[
   P_{f|X}(\varphi_t)\leq P_{\rm tree}(x,\varphi_t,V_0).
\]
This follows easily considering the contribution of the backward branches of~$f^{-n}(x)$ contained in~$X$ in the sum in~\eqref{def.tree.1}.

Let~$Y$ be a neighborhood of~$X$ on which~$f|_Y$ is uniformly expanding.
Thus, there is a constant~$C > 1$ and for every~$y \in Y$ there is~$x \in X$ such that for every integer~$n \ge 1$ and every~$x' \in f^{-n}(x)$ there is~$y' \in f^{-n}(y)$ shadowing~$x'$ and so that
$$ C^{-1} < \lvert (f^n)'(y')\rvert / \lvert (f^n)'(x')\rvert < C. $$
It follows that for each neighborhood~$V_0$ disjoint from~$Y$ we have~$P_{f|_X}(\varphi_t) \le P_{\rm tree}(y, \varphi_t, V_0)$.

By the eventually onto property of~$f$ on~$J$, we have $f^m(Y)=J$ for some $m \ge 1$.
Fix $y\in Y \cap f^{-m}(z)$ and let~$V_0$ neighborhood of~$\Sigma$ disjoint from~$Y$.
Then we have
\[\begin{split}
   P_{m+n}(z, \varphi_t,V_0)
&=
   \frac{1}{m+n}\log\sum_{x\in f^{-(m+n)}(z)\cap J\setminus V_0}\lvert (f^{m+n})'(x)\rvert^{-t}
\\ &\ge
   \frac{1}{m+n}\log\lvert (f^{m})'(y)\rvert^{-t}
+ \frac{n}{m+n} P_n(y, \varphi_t, V_0).
\end{split}\]
This shows $P_{\rm tree}(z,\varphi_t,V_0)\ge P_{\rm tree}(y,\varphi_t,V_0)$ and completes the proof of the inequality $P_{\rm tree}(z,\varphi_t,V_0) \ge \widetilde{P}_{\rm hyp}(\varphi_t)$.
\end{proof}

\section{$\sigma$-finite conformal measures}\label{sec:bal}

Recall that $f$ is a rational map~$f$ of degree at least~$2$.
If~$f$ is exceptional, then~$\Sigma$ is the maximal finite and
forward invariant subset of~$J$ satisfying $f^{-1}(\Sigma)
\setminus \Sigma \subset \Crit$. Otherwise~$\Sigma =
\emptyset$.

In the following proposition we adapt the classical method by Patterson and Sullivan to construct a $e^{\widetilde{P}(\varphi_t) - \varphi_t}$-conformal measure on~$J$ for each~$t < 0$. For a map without a phase transition in the negative spectrum or for a map with a phase transition in the negative spectrum at some parameter~$t_- < t$, we obtain a finite conformal measure supported on~$J$, as in part~$1$ of
Proposition~\ref{p:known conformal}.
For a map having a phase transition in the negative spectrum at some parameter $t_- > t$ this construction gives us a conformal measure outside~$\Crit$, which is finite outside each neighborhood of~$\Sigma$.
Recall that by part~$2$
of Proposition~\ref{p:known conformal}, existence of phase
transition implies that there does not exist a \emph{finite}
$e^{\widetilde{P}(\varphi_t) - \varphi_t}$\nobreakdash-conformal
measure for $t<t_-$.

\begin{proposition} \label{fsfjl}
    Let~$f$ be a rational function of degree at least~$2$.
For each $t < 0$ there exists a Borel measure on~$J$ that is $e^{\widetilde{P}(\varphi_t)-\varphi_t}$\nobreakdash-conformal outside~$\Crit$, finite outside any neighborhood of~$\Sigma$, gives zero measure to $\Sigma \cup \Crit$ and whose support is equal to~$J$.
\end{proposition}

\begin{proof}
As a first step we will apply the Patterson-Sullivan method while considering only those inverse branches outside a given neighborhood~$V$ of~$\Sigma$.
We obtain in this way a measure that is $e^{\widetilde{P}(\varphi_t) -
\varphi_t}$\nobreakdash-conformal outside the set~$\overline{V} \cup
f^{-1}(\overline{V}) \cup \Crit$. We will obtain a measure
$e^{\widetilde{P}(\varphi_t) - \varphi_t}$\nobreakdash-conformal
outside~$\Crit$ by taking the limit of the measures obtained by
repeating this construction with~$V$ replaced by smaller and
smaller neighborhoods.

We start with the following lemma.
Recall that $\Sigma_0$ denotes the set of neutral periodic points in~$\Sigma$ plus its preimages.

\begin{lemma}\label{l:zero pressure}
   Given $t<0$, for every~$\lambda > 0$ there exist positive numbers $r$ and $A$  such that for every $x \in J$ and every integer $\ell \ge 1$ we have
   \[
       \sum\, \lvert (f^\ell)'(y)\rvert^{-t}
       \le A \,e^{\ell \lambda},
   \]
   where the sum is taken over all $y \in f^{-\ell}(x)$ satisfying $f^j(y) \in B(\Sigma_0, r)$ for every $j \in \{0, \ldots, \ell - 1 \}$.
\end{lemma}

\begin{proof}
Let~$r > 0$ be sufficiently small so that for each periodic point~$p \in \Sigma_0$ of minimal period~$n \ge 1$ we have
\[
   \sup_{y\in B(p, r)} \lvert (f^n)'(y)\rvert^{-t}
   \le
   e^{\lambda n}.
\]
Hence, there is some constant~$A_0 > 0$ such for every integer~$\ell \ge 1$ and every point~$y$ satisfying $f^j(y) \in B(\Sigma_0, r)$ for every~$j \in \{ 0, \ldots, \ell - 1 \}$ we have
\[
   \lvert (f^\ell)'(y)\rvert^{-t} \le A_0\, e^{\ell \lambda}.
\]
Reducing~$r$ if necessary, we may assume that for every~$p \in \Sigma_0$ the map~$f$ is injective on~$B(p, r)$ and the set~$f(B(p, r))$ is disjoint from $B\big(\Sigma_0 \setminus \{ f(p) \}, r\big)$.
So for each~$p \in \Sigma_0$ and $w \in B(f(p), r)$ there is at most one point~$w' \in B(p, r)$ such that~$f(w') = w$.
By induction we can conclude that for each~$\ell \ge 1$, $x \in J$, and~$p$, $p' \in \Sigma_0$ there is at most one point~$y \in f^{-\ell}(x)$ such that
\[
   y \in B(p, r)\quad \text{ and }\quad f^{\ell - 1}(y) \in B(p', r).
\]
Thus the assertion follows with $A \eqdef A_0 (\card \Sigma_0)^2$.
\end{proof}

We now continue in proving the proposition.
Let~$z \in J \setminus \Sigma$ be an expanding safe point as provided by Lemma~\ref{lem:exisafe}.
Given $\lambda = \widetilde{P}(\varphi_t)/3$, let~$r$ and~$A$ be the positive numbers provided by Lemma~\ref{l:zero pressure}.
Reducing~$r > 0$ if necessary, we can assume that $z\notin B(\Sigma, r)$ and by Proposition~\ref{sngk} we can assume that $V \eqdef B(\Sigma,r)$ satisfies $P_{\rm tree}(z,\varphi_t,V) = \widetilde{P}(\varphi_t)$.

There exists a sequence $\{b_n\}_{n\ge 1}$ of positive reals such that
\begin{equation}\label{lzor}
   \sum_{n=1}^\infty b_ne^{-np}\sum_{x\in f^{-n}(z)\cap J\setminus V}
       \lvert (f^n)'(x)\rvert^{-t}
       \begin{cases}
       <\infty& \text{ if }p>\widetilde{P}(\varphi_t),\\
       =\infty&\text{ if }p\le \widetilde{P}(\varphi_t),
       \end{cases}
\end{equation}
and $\lim_{n\to\infty}b_n/b_{n+1}=1$ (see, for example,~\cite[Lemma 3.1]{DenUrb:91}).
Given $t<0$ and $p>\widetilde{P}(\varphi_t)$, let us define
\[
M_{t,p}\eqdef \sum_{n=1}^\infty b_ne^{-np} \sum_{x\in
f^{-n}(z)\cap J\setminus V}\lvert (f^n)'(x)\rvert^{-t}
\]
and for each neighborhood~$W$ of~$\Sigma$ define the measure
\begin{equation}\label{def.mu.dip}
\mu_{t,W,p}\eqdef
\frac{1}{M_{t,p}}\sum_{n=1}^\infty b_ne^{-np}
\sum_{x\in f^{-n}(z)\cap J\setminus W}\lvert (f^n)'(x)\rvert^{-t}\,\delta_x,
\end{equation}
where $\delta_x$ denotes the Dirac measure supported at $x$.

Observe that the measure $\mu_{t,V,p}$ is probabilistic for any $p>\widetilde{P}(\varphi_t)$.
If $W\subset V$ is a neighborhood of~$\Sigma$ the measure $\mu_{t,W,p}$ is not probabilistic in general, however it is finite as shown in the following lemma.
Let us denote by $\lvert\mu\rvert$ the total mass of a measure $\mu$, that is, let $\lvert\mu\rvert=\mu(J)$.

\begin{lemma} \label{lsls}
   For every neighborhood $W$ of $\Sigma$ contained in $V$ there is a positive   constant $C(W)$ such that for every $p > \widetilde{P}(\varphi_t)$ we have
\[
    1 \le \lvert\mu_{t,W,p}\rvert \le C(W).
\]
\end{lemma}
\begin{proof}
Since by assumption~$W \subset V$, we have~$\lvert\mu_{t,W,p}\rvert \ge \lvert\mu_{t,V,p}\rvert = 1$.

It only remains to prove the upper bound.
Put
\begin{equation}\label{e.asasb}
   V_1\eqdef V, \quad\text{ and for $i \ge 1$ put }\quad
   V_{i+1}\eqdef V_i \cap f^{-1}(V_i).
\end{equation}
Note that $V_{i+1}$ is the subset of $V$ consisting of all points that under forward iteration do not leave $V$ for at least $i$ steps.
In particular,
\begin{equation}\label{e.oldlem}
   f(V_i\setminus V_{i+1}) \subset \overline{\bC}\setminus V_i
\end{equation}
for all $i\geq 1$.

We will first consider a simple case and assume that we have~$V_{i_0}\subset W$ for some integer~$i_0 \ge 1$. We establish an upper bound for $\lvert \mu_{t,V_{i_0},p}\rvert$ and hence for $\lvert\mu_{t,W,p}\rvert$.
Observe first that for every $i\ge1$ we have
\[
   \lvert\mu_{t,V_{i+1},p}\rvert =
   \mu_{t,V_{i+1},p}(V_i\setminus V_{i+1}) + \mu_{t,V_{i+1},p}(J\setminus V_i)
   \le \mu_{t,V_{i+1},p}(V_i\setminus V_{i+1}) + \lvert\mu_{t,V_i,p}\rvert.
\]
By~\eqref{e.oldlem} for any point $x\in V_i\setminus V_{i+1}$
we have $f(x)\in \overline \bC\setminus V_i$. Hence, we can
estimate
\[\begin{split}
&M_{t,p}\cdot\mu_{t,V_{i+1},p}(V_i\setminus V_{i+1})\\
&=\sum_{n=1}^\infty b_{n}\,e^{-np}\sum_{x\in f^{-n}(z)\cap V_i\setminus V_{i+1}}
    \lvert (f^{n})'(x)\rvert^{-t}\\
&\leq
e^{-p} \deg f \, \sup_J\, \lvert f'\rvert^{-t}
\sum_{n=1}^\infty  b_{n}e^{-(n-1)p}
   \sum_{y\in f^{-(n-1)}(z)\cap J\setminus V_i}
   \lvert (f^{n-1})'(y)\rvert^{-t}\\
&\le
e^{-p} \deg f \, \sup_J\, \lvert f'\rvert^{-t}
\left(b_1 +
\max_{k\ge 2}\frac{b_k} {b_{k-1}}
   \sum_{n=1}^\infty b_n \, e^{-np}
   \sum_{x\in f^{-n}(z)\cap J\setminus V_i} \lvert (f^n)'(x)\rvert^{-t}
\right) \\
& =
e^{-p} \deg f \, \sup_J\, \lvert f'\rvert^{-t}
\left(b_1 +
\max_{k\ge 2}\frac{b_k} {b_{k-1}}
   M_{t,p}\cdot\lvert \mu_{t,V_i,p}\rvert
\right).
\end{split}\]
Since $\lim_{k\to\infty}b_k/b_{k+1}=1$, we have $ \max_{k\ge 2}b_k/b_{k-1} < \infty$.
So, if we put $C_0 = \deg f \cdot \max_{k\ge 2}b_k/b_{k-1}$ and $C_1 = \deg f \cdot b_1 \, M_{t, p}^{-1}$, we obtain
\[
\lvert\mu_{t,V_{i+1},p}\rvert
\le C_1 e^{-p} \, \sup_J\, \lvert f'\rvert^{-t} + \lvert\mu_{t,V_i,p}\rvert
   \left( 1 + C_0 \, e^{-p} \, \sup_J\, \lvert f'\rvert^{-t}
   \right).
\]
Now let~$C\eqdef 1 + C_1 C_0^{-1}$ and $C' \eqdef C_0$. Since $\lvert \mu_{t,V_1,p}\rvert=1$, we obtain by induction in~$i$ that
\begin{equation}\label{e:bound for V_i}
       \lvert \mu_{t,V_{i_0},p}\rvert
       \le C \left(1+C' \, e^{-p} \, \sup_J \,\lvert f'\rvert^{-t}\right)^{i_0-1}.
\end{equation}
Finally, recalling that $\lvert \mu_{t,W,p} \rvert \le \lvert \mu_{t,V_{i_0},p}\rvert $, this  proves the lemma in this first simple case that we considered.

Naturally, if $\Sigma=\Sigma_+$ then we can choose $V$ in such a way that $\bigcap_{i = 1}^{\infty} V_i = \Sigma$ and then it would be enough to consider the above case in which $V_i$ is eventually contained in $W$. However, if $\Sigma_0\ne\emptyset$, that is, if $\Sigma$ contains a neutral periodic point in~$J$ then this is not possible by the existence of Siegel compacta~\cite[Theorem 1]{PerMar:97}.

Let us now consider the general case. Recall that $V=B(\Sigma,r)$.
Let $W\subset V$ be an arbitrary neighborhood of~$\Sigma$.
Certainly we can take~$i \ge 1$ sufficiently large such that~$V_i\cap B(\Sigma_+,r)\subset W$. Increasing~$i$ if necessary, we can assume that for every integer~$k \ge i$ we have~$b_{k + 1}/b_k \le \exp(\widetilde{P}(\varphi_t)/3)$.
For each $x\in f^{-n}(z)\cap J\setminus W$ one of the following two cases can occur: Either a)  $x\notin V_i$ or b) $x\in V_i\setminus W$, and hence
\[
   M_{t,p}\cdot\lvert\mu_{t,W,p}\rvert
   \le M_{t,p}\cdot\lvert\mu_{t,V_i,p}\rvert
       +\sum_{n=1}^\infty\sum_{y\in f^{-n}(z)\cap J\cap V_i\setminus W}
            b_ne^{-np}\lvert (f^{n})'(y)\rvert^{-t}.
\]
In evaluating the latter term observe that to each point $x\in f^{-n}(z)\setminus V_i$ we may find some branch of preimages determined by a point $y\in f^{-\ell}(x)$ satisfying $f^j(y)\in V_i\setminus W$ and hence $f^j(y)\in B(\Sigma_0,r)$ for every $j\in\{0, \ldots, \ell-1\}$. However, by Lemma~\ref{l:zero pressure}, given any $x\in V_i$ the contribution of all such branches can be estimated by
\[\begin{split}
   \sum_{y\in f^{-\ell}(x)} b_{n+\ell}\,e^{-(n+\ell)p}
   &\lvert (f^{n+\ell})'(y)\rvert^{-t} \\
   &= b_n\, e^{-np} \lvert (f^n)'(x)\rvert^{-t} \cdot
       \frac{b_{n+\ell}}{b_n} e^{-\ell p}
       \sum_{y\in f^{-\ell}(x)} \lvert (f^{\ell})'(y)\rvert^{-t}\\
   &\le
       b_n e^{-np} \lvert (f^n)'(x)\rvert^{-t}
       \cdot e^{-\ell p}\max_{k\ge i}\frac{b_{k+\ell}}{b_k} \,A\, e^{\ell\lambda},
\end{split}\]
where each sum is taken over all $y$ such that $f^j(y)\in B(\Sigma_0,r)$ for every $j\in\{0, \ldots, \ell-1\}$.
Thus, summing over all such branches that could occur, by our previous choice of $i$ and $\lambda$ we can estimate
\[\begin{split}
   \lvert\mu_{t,W,p}\rvert
   &\le
   \lvert \mu_{t,V_i,p}\rvert
   \left( 1 +
   \sum_{\ell = 1}^{\infty} \Big( e^{-\ell p}
       \max_{k \ge i} \Big(\frac{b_{k + 1}}{b_k}\Big)^\ell
       A\,e^{\ell\lambda}\Big)\right)\\
   &\le
   \lvert\mu_{t,V_i,p}\rvert
       \left( 1 +
           A\, \sum_{\ell = 1}^{\infty}
           e^{-\ell p} e^{2\ell\widetilde{P}(\varphi_t)/3}\right)\\
   &\le
   \lvert \mu_{t,V_i,p}\rvert
   \left( 1 +
           A\, \sum_{\ell = 1}^{\infty}
           e^{- \ell \widetilde{P}(\varphi_t)/3} \right).
\end{split}\]
Note that $\widetilde P(\varphi_t)>0$.
Together with~\eqref{e:bound for V_i} this completes the proof of the lemma.
\end{proof}
\begin{lemma}
Given a neighborhood~$W$ of~$\Sigma$ contained in~$V$, as $p\searrow \widetilde{P}(\varphi_t)$ there exists a non-zero finite measure that is a weak* accumulation
point of the family of measures $\{\mu_{t,W,p}\colon
p>\widetilde{P}(\varphi_t) \}$.  Furthermore, each such measure is
$e^{\widetilde{P}(\varphi_t) - \varphi_t}$-conformal
outside the set~$\overline{W}\cup f^{-1}(\overline{W})\cup {\rm
Crit}$.
\end{lemma}

\begin{proof}
First observe that, by Lemma~\ref{lsls} the total mass of any of
the measures in $\{\mu_{t,V_i,p}\colon p>\widetilde{P}(\varphi_t)
\}$ is uniformly bounded from above and below  by some positive
constant. Hence this family of measures is relatively compact in
the weak* topology and thus possesses a non-zero and finite
accumulation point proving the first claim.

Following the construction in~\cite[Section 3]{DenUrb:91}, we have for every special set~$A$ disjoint from~$W \cup f^{-1}(W) \cup {\rm Crit}$
    \begin{equation}\label{uusp}\begin{split}
    \mu_{t,W,p}(f(A))
    &=\frac{1}{M_{t,p}}
    \sum_{n=1}^\infty \sum_{y\in  f(A)\cap f^{-n}(z)} b_n e^{S_n\varphi_t(y)-np}\\
    &=\frac{1}{M_{t,p}}
    \sum_{n=1}^\infty \sum_{x\in A\cap f^{-(n+1)}(z)} b_n e^{S_n\varphi_t(f(x))-np}\\
    &=\frac{1}{M_{t,p}}
    \sum_{n=1}^\infty \sum_{x\in A\cap f^{-(n+1)}(z)} b_n e^{S_{n+1}\varphi_t(x)-(n+1)p}
        e^{p-\varphi_t(x)}.
    \end{split}\end{equation}
Thus,
    \[\begin{split}
    &\Delta_A(t,W,p)\eqdef
    \Big\lvert \mu_{t,W,p}(f(A)) - \int_Ae^{\widetilde{P}(\varphi_t)-\varphi_t}\,d\mu_{t,W,p}\Big\rvert\\
    &=\frac{1}{M_{t,p}}\Big\lvert
    \sum_{n=1}^\infty \sum_{x\in A\cap f^{-(n+1)}(z)} e^{S_{n+1}\varphi_t(x)-(n+1)p}
        e^{-\varphi_t(x)} \left[b_ne^p-b_{n+1}e^{\widetilde{P}(\varphi_t)}\right]\\
    &\phantom{=\frac{1}{M_{t,p}}\Big\lvert}
        - b_1\sum_{x\in A\cap f^{-1}(z)}e^{\widetilde{P}(\varphi_t)-p}\Big\rvert   \\
    &\le\frac{1}{M_{t,p}}\sum_{n=1}^\infty\sum_{x\in A\cap f^{-(n+1)}(z)}
        b_{n+1}\Big\lvert \frac{b_n}{b_{n+1}}-e^{\widetilde{P}(\varphi_t)-p}\Big\rvert
        e^{p-\varphi_t(x)}
        e^{S_{n+1}\varphi_t(x)-(n+1)p}\\
    &\phantom{=\frac{1}{M_{t,p}}\Big\lvert}
        + \frac{1}{M_{t,p}}b_1 \deg f \cdot e^{\widetilde{P}(\varphi_t) - p}.
    \end{split}\]
Recall that, by the choice of $\{b_n\}_{n \ge 1}$
in~\eqref{lzor}, we have $\lim_{n\to\infty}b_n/b_{n+1}=1$ and
$\lim_{p\searrow \widetilde{P}(\varphi_t)}M_{t,p}=\infty$. Hence,
we obtain $\lim_{p\searrow
\widetilde{P}(\varphi_t)}\Delta_A(t,W,p)=0$ uniformly in $A$. The
assertion now follows like in~\cite[Section 3]{DenUrb:91} (see
also Section~12.1 or Lemma~12.5.5 and Remark~12.5.6
in~\cite{PrzUrb:}). This proves the lemma.
\end{proof}

We are now prepared to finish the proof of the proposition. Note
that in~\eqref{def.mu.dip} we use the same normalization factor
$M_{t,p}$ for all measures $\mu_{t,W,p}$ for any neighborhood $W$. Hence given
$p>\widetilde{P}(\varphi_t)$ for any pair of neighborhoods~$W$ and~$W'$ of~$\Sigma$ such that~$W' \subset W \subset V$ we have
\begin{equation}\label{koss}
   {\mu_{t,W',p}}|_{J\setminus \overline{W}}
   = {\mu_{t,W,p}}|_{J\setminus \overline{W}}.
\end{equation}
Using a diagonal argument we can
conclude that, as $p \searrow \widetilde{P}(\varphi_t)$ and
$\rho \to 0$, there exists a weak* accumulation measure $\nu_t$ of
the family
\[
   \{ \mu_{t,B(\Sigma, \varepsilon),p} \colon t > \widetilde{P}(\varphi_t),
       \varepsilon \in (0, r) \}
\]
and that~$\nu_t$ is $e^{\widetilde P(\varphi_t)-\varphi_t}$-conformal outside
$\Sigma \cup\Crit$. Replacing~$\nu_t$ by the restricted measure
$\nu_t|_{J\setminus(\Sigma\cup\Crit)}$, if necessary we can assume
that~$\nu_t$ does not give weight to~$\Sigma \cup \Crit$ and hence that~$\nu_t$ is conformal outside~$\Crit$.

Lemma~\ref{lsls} and~\eqref{koss} together imply that then~$\nu_t$ is finite outside each neighborhood of~$\Sigma$.

Finally, the fact that the support of~$\nu_t$ is equal to~$J$ follows from the property that~$f$ is locally eventually onto on~$J$.
This finishes the proof of Proposition~\ref{fsfjl}.
\end{proof}

\section{Proof of the main result}\label{sec:final}

In this section we prove Theorem~\ref{main1}.
In Section~\ref{sec:final-1} we make use of the bridges construction in~\cite{GelPrzRam:} to prove the lower bound, which follows along the same lines as in \cite[Sections 2.2, 2.3, 5]{GelPrzRam:}.
We point out that, after a careful observation, in fact it applies without any changes to our present more general setting.
In Section~\ref{sec:final-1} we also give another application of the bridges construction (Lemma~\ref{help} and its Corollary~\ref{def}), which is used in Section~\ref{sec:final-3}.
The upper bound is shown in Section~\ref{sec:final-2}, where in the case~$t < 0$ we use the $\sigma$-conformal measure given by Proposition~\ref{fsfjl}.
We complete the proof of Theorem~\ref{main1} in Section~\ref{sec:final-3} by showing that the upper Lyapunov exponent of each point in~$J \setminus \Sigma$ is at most~$\widetilde{\alpha}^+$.
\subsection{Lower bound}\label{sec:final-1}

We refer the reader to~\cite{GelPrzRam:} for all the notation in this subsection.

We call a point $x\in J$ \emph{non-immediately post-critical} if there exists some preimage branch $x_0=x=f(x_1)$, $x_1=f(x_2)$, $\ldots$ that is dense in $J$ and disjoint from $\Crit$.
There are at most finitely many non-immediately post-critical points.
On the other hand, it is easy to see that each periodic point not in~$\Sigma$ is non-immediately post-critical.
It follows that every uniformly expanding set disjoint from~$\Sigma$ contains at least one non-immediately post-critical point.

A set $\Lambda$ is called $f$-\emph{uniformly expanding Cantor repeller} (ECR) if it is a uniformly expanding repeller and limit set of a finite graph directed system satisfying the strong separation condition with respect to $f$.

The following proposition generalizes~\cite[Proposition 1]{GelPrzRam:}.
Given an integer~$a \ge 1$ and a function~$g$ defined on~$J$ put
\[
   S_a(g) \eqdef g + g \circ f + \cdots + g \circ f^{a - 1}.
\]

\begin{proposition}\label{p.asympt}
    There exists a sequence $\{a_m\}_{m \ge 1}$ of positive integers and a sequence $\{ \Lambda_m \}_{m \ge 1}$ of subsets of~$J \setminus \Sigma$, such that for each~$m$ the set~$\Lambda_m$ is $f^{a_m}$-invariant and uniformly expanding topologically transitive set, in such a way that for every $t\in\mathbb{R}$ we have
    \[
        \widetilde{P}(\varphi_t) = \lim_{m\to\infty} \frac 1 {a_m}
        P_{f^{a_m}|\Lambda_m}(S_{a_m}\varphi_t)
        =
        \sup_{m \ge 1} \frac 1 {a_m}
        P_{f^{a_m}|\Lambda_m}(S_{a_m}\varphi_t)\,.
    \]
\end{proposition}

\begin{proof}
Recall the definition of the hidden hyperbolic pressure $\widetilde P_{\rm hyp}(\varphi_t)$ in~\eqref{def.hiddhyp}. By Propositions~\ref{p.primeiro} and~\ref{sngk} this pressure coincides with $\widetilde{P}(\varphi_t)$ and is obtained by taking a supremum over uniformly expanding repellers.
Note that, given $t\in\bR$ and $\varepsilon>0$ and a uniformly expanding repeller $\Lambda$, by \cite[Lemma 3]{GelPrzRam:} there exists a positive integer $n$ and an $f^n$-ECR $\Lambda'\subset\Lambda$ such that
\[
   \frac 1 n P_{f^n|\Lambda'}(S_n\varphi_t) \ge P_{f|\Lambda}(\varphi_t) -\varepsilon.
\]
Note that in this case $\Lambda'\cap\Sigma=\emptyset$ since $\Lambda\cap\Sigma=\emptyset$. Hence  $\Lambda'$ contains non-immediately post-critical points.
Note further that, given any two $f$\nobreakdash-ECR's $\Lambda_1$ and $\Lambda_2$ that both contain non-immediately post-critical points, by~\cite[Lemma 2]{GelPrzRam:} there exists an $f$-ECR $\Lambda\subset J \subset \Sigma$ containing $\Lambda_1\cup \Lambda_2$ and thus with pressure at least equal to the maximum of pressures of $\Lambda_1$ and $\Lambda_2$.

Based on these arguments, we can conclude that for any $N>0$ and $\varepsilon>0$
we can find an integer~$n \ge 1$ and a topologically transitive $f^n$-ECR $\Lambda\subset J$ so that
\[
 \frac 1 n P_{f^n|\Lambda}(\varphi_t)\ge \widetilde{P}(\varphi_t) - \varepsilon
\]
for all $t\in (-N,N)$. This proves the proposition.
\end{proof}

The existence such an approximating sequence of repellers and~\cite[Theorem 3]{GelPrzRam:} together imply the following estimate, which is part of Theorem~\ref{main1}.

\begin{proposition}
   For $\alpha^- \leq \alpha \leq \beta \leq \widetilde{\alpha}^+$ we have
   \[
       \dim_{\rm H} \cL(\alpha, \beta)
       \geq \min\{\widetilde F(\alpha),\widetilde F(\beta)\}\,.
   \]
\end{proposition}

\begin{proof}
   Consider now a family $\{\Lambda_m\}_{m\ge1}$ of $f^{a_m}$-ECRs as provided by Proposition~\ref{p.asympt} and assume that the spectrum of Lyapunov exponents of $\Lambda_m$ eventually contains any exponent in $(\alpha^-,\widetilde{\alpha}^+)$. Given $\alpha$, $\beta\in[\alpha^-,\widetilde{\alpha}^+]$ we can choose a sequence $\{\gamma_m\}_{m\ge1}$ so that $\liminf_{m\to\infty}\gamma_m=\alpha$ and $\limsup_{m\to\infty}\gamma_m=\beta$ and that each $\gamma_m$ is a Lyapunov exponent of $f^{a_m}|_{\Lambda_m}$. Recall that hence for each $m$ there exist a unique number $t_m=t_m(\gamma_m)\in\bR$ so that $a_m\,\gamma_m=-\frac{d}{ds}P_{f^{a_m}|\Lambda_m}(S_{a_m}\varphi_s)|_{s=t_m}$.
Moreover,  there exists an equilibrium state $\mu_m$ for the potential $S_{a_m}\varphi_{t_m}$ with respect to $f^{a_m}|_{\Lambda_m}$ with Lyapunov exponent (with respect to $f^{a_m}$) equal to $a_m\gamma_m$ and satisfying
\[\begin{split}
   \dim_{\rm H}\mu_m
   &= \frac{h_{\mu_m}(f^{a_m})}{a_m \chi(\mu_m)}
   = \frac{P_{f^{a_m}|\Lambda_m}(\varphi_{t_m})+t_m\, a_m\gamma_m}
       {a_m\gamma_m}\\
   &\ge \frac{1}{a_m \gamma_m}\inf_{t\in\bR}
       \big(P_{f^{a_m}|\Lambda_m}(\varphi_t)+t\, a_m \gamma_m\big)
   \eqdef F_{f^{a_m}|\Lambda_m}(\gamma_m).
\end{split}\]
By Proposition~\ref{p.asympt} and~\eqref{e.desFtilde}, we can conclude that $F_{f^{a_{m_k}}|\Lambda_{m_k}}(\gamma_{m_k})\to\widetilde F(\alpha)$ if $\gamma_{m_k}\to\alpha$ and $F_{f^{a_{m_k}}|\Lambda_{m_k}}(\gamma_{m_k})\to\widetilde F(\beta)$ if $\gamma_{m_k}\to\beta$.
Together with~\cite[Theorem~3]{GelPrzRam:}, this proves the proposition.
\end{proof}

There is one more useful application of the bridges construction.

\begin{lemma} \label{help}
   Given an expanding periodic point $p\notin\Sigma$ and $\varepsilon>0$, there exist a uniformly expanding repeller $\Lambda$ disjoint with $\Sigma$ of positive Hausdorff dimension, containing $p$, and an ergodic non-atomic measure $\mu$ supported on~$\Lambda$ such that
   \[
       \lvert \chi(\mu) - \chi(p)\rvert <\varepsilon\,.
   \]
\end{lemma}

\begin{proof}
We start from the orbit $P$ of the periodic point $p$. As $p\notin\Sigma$, $p$ is non-immediately post-critical. Hence, we can find a bridge from $P$ going back to $P$ and construct $\Lambda$ as in \cite[Lemma 2]{GelPrzRam:}. We can then distribute a Gibbs measure on $\Lambda$ choosing potential in such a way that probability of the backward branch going through the bridge is very small.
\end{proof}

\begin{corollary} \label{def}
   In the definition of $\widetilde{\alpha}^+$, instead of nonatomic ergodic measures one can use ergodic measures with support outside $\Sigma$ or ergodic measures giving measure zero to $\Sigma$.
\end{corollary}

\subsection{Upper bound}
\label{sec:final-2}

Recall that a point $x$ is called \emph{conical} if there exists a
number $r(x)>0$, a sequence of numbers
$n_\ell=n_\ell(x)\nearrow\infty$, and a sequence
$U_\ell=U_\ell(x)$ of neighborhoods of $x$ such that
$f^{n_\ell}(U_\ell)= B(f^{n_\ell}(x),r)$, the map  $f^{n_\ell}$ is
univalent on $U_\ell$ and that distortion $\Dist
f^{n_\ell}|_{U_\ell}$ is bounded uniformly in~$\ell$ and~$x$ by a constant~$K > 1$ (the latter condition follows from the former one from Koebe's distortion lemma by replacing~$r$ by say~$r/2$).

The following proposition will allow us to restrict our
considerations concerning dimension to conical points with
positive exponents.

\begin{proposition}[{\cite[Proposition 3]{GelPrzRam:}}]\label{conic1}
   The set of points $x\in J$ that are not conical and satisfy $\overline\chi(x)>0$ has Hausdorff dimension zero.
\end{proposition}

We are now ready to prove an upper bound for the dimension.

\begin{proposition}\label{prop:upper}
    Let $0<\alpha\le\beta \le\widetilde\alpha^+$.
    We have
    \[
    \dim_{\rm H}\cL(\alpha,\beta)\le \max \left\{ 0,\max_{\alpha\le q\le \beta}\widetilde F(q) \right\}.
    \]
\end{proposition}

\begin{proof}
The proof will follow the same ideas as the proof
of~\cite[Proposition 2]{GelPrzRam:}. The only difference is that in the case~$f$ has a phase transition in the negative spectrum and~$t < t_{-}$, we will use a $\sigma$-finite conformal measure constructed in Section~\ref{sec:bal} instead of a conformal probability measure.

By Proposition~\ref{conic1} it is sufficient to study the subset $\cL_{\rm c}(\alpha,\beta)\subset\cL(\alpha,\beta)$ of points that are conical.
Recall that, by the Frostman Lemma, if there exist a finite Borel measure $\mu$ and a number $\theta$ such that for every $x\in\cL_{\rm c}(\alpha,\beta)$ we have
\[
   \underline d_\mu(x)\eqdef \liminf_{\delta\to0}\frac{\log\mu(B(x,\delta))}{\log\delta}\le \theta
\]
then $\dim_{\rm H}\cL_{\rm c}(\alpha,\beta)\le \theta$ (see also~\cite[Theorem
7.2]{Pes:96}).

Given a conical point $x\in J\setminus \Sigma$ with $0<\alpha=\underline\chi(x)$ and $\beta=\overline\chi(x)$, there exist numbers $q(x)\in[\alpha,\beta]\setminus \{0\}$, $r(x)>0$, and $K(x)>1$,
and a sequence of numbers $n_\ell=n_\ell(x)$ such that
    \begin{equation} \label{eqn:gh}
        \lim_{\ell\to\infty}\frac{1}{n_\ell}\log\,\lvert (f^{n_\ell})'(x)\rvert = q(x)
    \end{equation}
 and that
    \begin{equation} \label{est2}
        r\,\lvert(f^{n_\ell})'(x)\rvert^{-1} K(x)^{-1}
        \le \diam f^{-n_\ell}_x\big(B(f^{n_\ell}(x),r)\big)
        \le r \,\lvert(f^{n_\ell})'(x)\rvert^{-1}K(x)
    \end{equation}
for all $\ell$ and all $r\in(0,r(x))$ (compare, for example,~\cite[Lemma
7]{GelPrzRam:}).
By omitting finitely many $n_\ell$ we can assume
that the right hand side of \eqref{est2} is not greater than $r$.
Replacing $n_\ell$ by $n_\ell - 1$ and $r(x)$ by $r(x)/\sup\,\lvert f'\rvert$
if necessary, we can also freely assume that
$B(f^{n_\ell}(x), r(x))\cap \Crit=\emptyset$.

Given $x\in\cL_{\rm c}(\alpha,\beta)$, let us fix  $r > 0$ satisfying
\[
   r = \frac 1 2 \min\left\{r(x), \dist(x,\Crit\cup\,\Sigma)\right\}.
\]
Denote
\[
   U_{\ell}\eqdef f_x^{-n_\ell}(B(f^{n_\ell}(x),r)).
\]
Observe that $B(f^{n_\ell}(x), 2\,r)$ does not intersect
$\Sigma\cup \Crit$. Indeed, by our assumption it does not
intersect $\Crit$. Further, if it intersected $\Sigma$ then either
$f^n(f^{-n_\ell}_x(B(f^{n_\ell}(x),2\,r)))$ would intersect
$\Crit$ for some $0\leq n<n_\ell$ or
$f^{-n_\ell}_x(B(f^{n_\ell}(x),2\,r))$ would intersect $\Sigma$.
The former is impossible because the map would not be univalent
there (we remind that $2\,r<r(x)$), the latter is impossible
because $\dist(x,\Sigma)>2\,r>\diam
f^{-n_\ell}_x(B(f^{n_\ell}(x),2\,r))$.

Given $t\in\bR$ let $\mu_t$ be the $\sigma$-finite measure that is $\exp(\widetilde P(\varphi_t)-\varphi_t)$-conformal outside~$\Crit$ as provided by Proposition~\ref{fsfjl} and if $t\ge 0$ then let~$\mu_t$ be the finite $\exp(\widetilde P(\varphi_t)-\varphi_t)$-conformal measure provided by Proposition~\ref{p:known conformal}.
Since in all cases~$\mu_t$ is finite outside every neighborhood of~$\Sigma$, if we put $\kappa = \mu_t (J \setminus B(\Sigma, r))$ then~$\kappa$ is finite and for every~$\ell$ we have,
\begin{equation}\label{huch}
  \mu_t(U_\ell)
  \le
    \kappa \, K^t e^{-n_\ell \widetilde P(\varphi_t)}\,\lvert
    (f^{n_\ell})'(x)\rvert^{-t}.
\end{equation}
As~$x$ is conical, we have
\[
   U_\ell\supset B\big(x, K^{-1} r \,\lvert(f^{n_\ell})'(x)\rvert^{-1}\big).
\]
Together with~\eqref{huch} and~\eqref{eqn:gh} this yields
\begin{multline*}
       \underline d_{\mu_t}(x)
   =
   \liminf_{\delta\to0}\frac{\log\mu_t(B(x,\delta))}{\log\delta}
   \\ \le
   \liminf_{\ell\to0} \frac{\log\mu_t\big(B(x,K^{-1}r\,\lvert (f^{n_\ell})'(x)\rvert^{-1}\big)} {\log\,(K^{-1}r\,\lvert (f^{n_\ell})'(x)\rvert^{-1})}
   \le
   \frac{\widetilde P(\varphi_t)+t\,q}{q}.
\end{multline*}
Recall that this is true for every $x\in\cL_{\rm c}(\alpha,\beta)$ and~$t \in \mathbb{R}$.
Now concluding as in the proof of ~\cite[Proposition
2]{GelPrzRam:}, using the Frostman lemma, we obtain that $\dim_{\rm
H}\cL_{\rm c}(\alpha,\beta)\le \max \left\{ 0, \max_{\alpha\le q\le\beta}\widetilde F(q) \right\}$, as wanted.
\end{proof}

\subsection{Completeness of the spectrum}
\label{sec:final-3}

We establish the following gap in the spectrum of upper exponents.

\begin{proposition}\label{prop:monic}
    If $f$ is exceptional then
    $\displaystyle \overline\chi(x)\le \widetilde\alpha^+$ for every $x\in J\setminus\Sigma$.
\end{proposition}

We give two proofs of this proposition, one in this section and the other one in Appendix~\ref{s:alternative}.

We will denote the spherical distance by~$\dist$.
Recall that for a rational map~$g$ and a critical point~$c \in \overline{\bC}$ of~$g$ we denote by~$\deg_g(c)$ the local degree of~$g$ at~$z = c$.

	By Corollary~\ref{def} we have $\widetilde\alpha^+<\alpha^+$ if and only if there is a periodic point in~$\Sigma$ whose exponent is strictly larger than $\widetilde\alpha^+$.
Hence, to prove Proposition~\ref{prop:monic} we need to control the exponent of any point $x \in J \setminus \Sigma$ whose orbit stays most of the time close to $\Sigma$.
Any orbit piece that shadows some (periodic) orbit in $\Sigma$ for a long time inherits its exponent, however right before it must have passed close to some critical point which results in a drop of the exponent.

Let us make this more precise.
For~$c \in f^{-1}(\Sigma) \setminus \Sigma \subset \Crit$ let~$k \ge 1$ be the minimal integer such that~$f^k(c)$ is a periodic point, put
\[
   \chi_{\ess}(c) \eqdef
   \frac{\chi(p)}{\deg_{f^k}(c)}.
\]
and if~$f$ is exceptional then we put
\[ \chi_{\ess}^+ \eqdef \max_{c \in f^{-1}(\Sigma) \setminus \Sigma}\chi_{\ess}(c).\]

We remark that there are examples where there is a point $c\in\Crit$ so that $f^k(c)\in \Sigma$ for some minimal number $k>1$.  If $f^k(c)\in\Sigma_0$ then $\chi_{\rm ess}(c) = 0$. If $f^k(c) \in \Sigma_+$ then $c' = f^{k -1}(c)$ is a critical point in $f^{-1}(\Sigma)$ and $\chi_{\rm ess}(c) < \chi_{\rm ess}(c')$. Thus, in none of these cases the ``essential exponent'' of $c$ and hence of an orbit piece that would shadow some periodic orbit $\{f^j(c)$, $j\ge k\}$ in $\Sigma$ would have large exponent. Hence, in what follows we can restrict ourselves to the case that $k=1$.

We have the following result.

\begin{lemma}\label{l.neu}
    Suppose~$f$ is exceptional. Let~$c \in f^{-1}(\Sigma_+) \setminus \Sigma_+$  and let~$k \ge 1$ be the minimal integer such that~$f^k(c) \in \Sigma_+$ is periodic. Then there exist constants~$\delta > 0$ and $C > 0$ such that for every $x \in J$ near~$c$, but different from~$c$, and every integer~$n \ge k$ such that for every~$j \in \{k, k + 1, \ldots, n - 1 \}$ we have~$f^j(x) \in B(\Sigma, \delta)$, the following estimate holds
\[
    \log \,\lvert (f^n)'(x)\rvert
    \le
    n \chi_{\ess}(c) + C.
\]
\end{lemma}

\begin{proof}
Put~$d \eqdef \deg_{f^k}(c)$. To prove the lemma, it suffices to
notice that if~$\delta$ is sufficiently small then for any orbit
piece $y$, $f(y)$, $\ldots$, $f^m(y)$ that stays $\delta$-close to
the periodic orbit of~$p = f^k(c)$ we have $\lvert (f^m)'(y)
\rvert \sim e^{m\chi(p)} = e^{m d \chi_{\ess}(c)}$. Thus
\[ \dist(f^k(x), f^k(c))
= \mathcal{O} \left( |(f^{n - k})'(p)|^{-1} \right)
= \mathcal{O} \left( e^{- n d\chi_{\ess}(c)} \right). \]
On the other hand, $\dist(x, c) \sim \dist(f^k(x), f^k(c))^{1/d}$, so
\[
    \lvert (f^k)'(x)\rvert
    \sim
    \dist(f^k(x), f^k(c))^{(d - 1)/d}
    =
    \mathcal{O} \left( e^{- n (d - 1) \chi_{\ess}(c)} \right),
\]
and $\lvert (f^n)'(x)\rvert = \mathcal{O} \left( e^{n\chi_{\ess} (c)} \right)$.
\end{proof}

Given a subset~$V$ of~$\overline{\bC}$, let
   \begin{equation}\label{e.sup}
       \chi^+(J\setminus V)\eqdef
       \limsup_{n\to\infty}\sup_{x\in J\setminus V}
       \frac 1 n \log\,\lvert (f^n)'(x)\rvert.
   \end{equation}

\begin{lemma}\label{l.heuleni}
If~$f$ is non-exceptional then there is an ergodic measure~$\mu$ supported on~$J$ and such that~$\chi(\mu) = \chi^+(J)$.
If~$f$ is exceptional and~$V$ is a neighborhood of~$\Sigma$, then one of the following cases holds:
   \begin{enumerate}
       \item either there exists an ergodic measure~$\mu$ such that
       \[   \mu(\Sigma) = 0
            \quad\text{ and }
        \quad\chi(\mu)\ge\chi^+(J\setminus V),
    \]
       \item or \[\chi^+(J\setminus V) \le \chi_{\ess}^+.\]
   \end{enumerate}
\end{lemma}

\begin{proof}
Let~$V$ be empty if~$f$ is not exceptional and let~$V$ be a neighborhood of~$\Sigma$ otherwise.
Without loss of generality in the latter case we can assume that $V$ is open.
Let~$\delta > 0$ be given by Lemma~\ref{l.neu}.
For each $n\ge1$ let $x_n\in J\setminus V$ be a point satisfying
\[
\frac 1 n \log\,\lvert (f^n)'(x_n)\rvert
=
\sup_{x\in J\setminus V} \frac 1 n \log\,\lvert (f^n)'(x)\rvert
\]
and consider the probability measure
\[
   \mu_n\eqdef \frac 1 n \sum_{k=0}^{n-1}\delta_{f^k(x_n)}.
\]
Consider a measure~$\mu$ that is accumulated by the sequence of
measures $\{\mu_n\}_{n\ge1}$ in the weak* topology. Notice
that~$\mu$ is $f$-invariant and satisfies $\chi(\mu)\ge\chi^+(J\setminus V)$.
It follows that there is a $f$-invariant an ergodic measure~$\mu'$
such that $\chi(\mu')\ge\chi^+(J\setminus V)$. If~$f$ is not exceptional or
if~$f$ is exceptional and~$\mu'(\Sigma) = 0$, then we are done.

To prove the remaining case, assume that~$f$ is exceptional and $\supp \mu' \subset \Sigma$.
Fix~$\varepsilon > 0$ and let~$\delta > 0$ and~$C > 0$ be given by
Lemma~\ref{l.neu}. Augmenting~$C> 0$ and reducing~$\delta$ if
necessary we can assume that~$B(\Sigma, \delta) \subset V$ and that for
each~$x \in J$ and each integer~$\ell \ge 1$ so that $f^j(x) \in
B(\Sigma_0, \delta)$ for every~$j \in \{ 0, \ldots, \ell - 1\}$, we have
\[
    \lvert(f^{\ell})'(x)\rvert
    \le \exp \left( \ell \varepsilon  + C \right).
\]
Let $V_+$ be a neighborhood of $\Sigma_+$ that is contained in
$B(\Sigma_+, \delta)$ and 
such that all preimages of a point in $V_+$ are either in $V_+$ or close to
$f^{-1}(\Sigma_+)\setminus \Sigma_+$. For an integer~$n \ge 1$ put
$$ N_n \eqdef \left\{ j \in \{0, \ldots, n - 1 \} \colon f^j(x_n) \in B(\Sigma_0, \delta) \right\} $$
and
\[
M_n \eqdef \left\{ j \in \{0, \ldots, n - 1 \} \colon f^j(x_n) \in V_+
\right\}.
\]
We have~$\lim_{n \to + \infty} (M_n+N_n) / n = 1$. Fix~$n$ and
let~$k \ge 1$ and let $j_1$, $\ldots$, $j_k$ be all the integers $j
\in \{1, \ldots, n - 1 \}$ such that~$f^{j - 1}(x_n) \notin V_+$
and~$f^j(x_n) \in V_+$. Similarly, let $k'$ be the number of
blocks of trajectory of $x$ contained in $B(\Sigma_0, \delta)$.
For each~$i \in \{1, \ldots, k \}$ let~$j_i'$ be the largest
integer~$j \in \{ j_i, \ldots, n - 1 \}$ such that for each~$s \in
\{j_i, \ldots, j \}$ we have~$f^s(x_n) \in B(\Sigma, \delta)$.
Then
\[
    \sum_{i = 1}^k (j_i' - j_i+1) = M_n
    \quad\text{ and }\quad
    \max \{k,k'\} \le n - (M_n +N_n).
\]
Furthermore, for each~$i \in \{1,
\ldots, k \}$ so that~$f^{j_i}(x_n) \in V_+$ the point~$f^{j_i -
1}(x_n)$ is close to~$f^{-1}(\Sigma)$ and we thus have
$$ \log \,\lvert(f^{j_i' - j_i + 1})'(f^{j_i - 1}(x_n))\rvert
\le (j_i' - j_i + 2) \chi_{\ess}^+ + C. $$ Hence, for
$C'=2C+\log\sup_J |f'|$ we have
\[\begin{split}
\log \,\lvert(f^n)'(x_n)\rvert &\le (M_n + k) \chi_{\ess}^+ + k C
+ N_n \varepsilon + k' C  \\
&\phantom{\le}\quad+(n-M_n-N_n-k) \log \sup_{J}\, \lvert f'\rvert
\\
&\le
n \max \{ \chi_{\ess}^+, \varepsilon \} + (n - M_n-N_n) C'.
\end{split}\]
Since~$\varepsilon > 0$ is arbitrary, this implies
$$ \chi^+(J\setminus V)
=
\limsup_{n \to + \infty} \frac{1}{n} \log \,\lvert(f^n)'(x_n)\rvert
\le
\chi_{\ess}^+ $$
finishing the proof of the lemma.
\end{proof}

\begin{lemma}\label{l:periodicexcep}
    Suppose~$f$ is exceptional. Then for each~$c \in f^{-1}(\Sigma_+) \setminus \Sigma_+$ and~$\varepsilon > 0$ there is a periodic point~$q$ close to~$c$ such that
    $$ \chi(q) \ge \chi_{\ess}(c) - \varepsilon. $$
    In particular, $\widetilde{\alpha}^+ \ge \chi_{\ess}^+$.
\end{lemma}

\begin{proof}
Let~$k \ge 1$ be the least integer such that~$p = f^k(c) \in \Sigma$ is periodic, put~$d \eqdef \deg_{f^k}(c)$ and let~$\ell \ge 1$ be the period of~$p$.
Let~$\delta > 0$ be sufficiently small so that there is a local inverse~$\phi$ of~$f^\ell$ fixing~$z = p$ defined on~$B(p, \delta)$, in such a way that~$\phi(B(p, \delta))$ is compactly contained in~$B(p, \delta)$ and for some constant~$\gamma_0 > 0$ and every~$n \ge 1$ and~$x \in \phi^n(B(p, \delta))$ we have~$|(f^{n \ell})'(x)| \ge \gamma_0 e^{n \ell \chi(p)}$.
Since~$c \not \in \Sigma$ there is a point~$x \in B(p, \delta)$ and an integer~$m \ge 1$ such that~$f^m(x) = c$ and such that~$(f^m)'(x) \neq 0$.
Let~$\rho > 0$ be sufficiently small so that the connected component~$W$ of~$f^{-m}(B(c, \rho))$ containing~$x$ is such that~$\overline{W} \subset B(p, \delta) \setminus \{ p \}$ and~$\gamma_1 = \inf_{z \in W} |(f^m)'(z)| > 0$.
Then for every sufficiently large integer~$n \ge 1$ there is a connected component~$W_n$ of~$f^{-k}(\phi^n(W))$ compactly contained in~$B(c, \rho)$.
It follows that~$W_n$ contains a periodic point~$q_n$ of~$f$ of period~$k + n \ell + m$.
We will now estimate its Lyapunov exponent.
Since~$\dist(\phi^n(W), p) \sim e^{-n \ell \chi(p)}$ we have $\dist(W_n, c) \sim e^{-n \ell \chi(p)}$, so there is a constant~$\gamma_2 > 0$ such that for every~$z \in W_n$ we have
$$ |(f^k)'(z)| \ge \gamma_2 e^{-n \ell \chi(p) (d - 1)/d}. $$
Therefore
$$ |(f^{k + n \ell + m})'(q_n)|
\ge
\gamma_0 \gamma_1 \gamma_2 e^{-n \ell \chi(p)/d}
=
\gamma_0 \gamma_1 \gamma_2 e^{-n \ell \chi_{\ess}(c)}, $$
and~$\liminf_{n  \to \infty} \chi(q_n) \ge \chi_{\ess}(c)$.
\end{proof}

\begin{proof}[Proof of Proposition~\ref{prop:monic}]
	In view of Lemma~\ref{l.heuleni} and Corollary~\ref{def}, the proposition is a direct consequence of the Lemma~\ref{l:periodicexcep}.
\end{proof}

We finally state the following immediate corollary.

\begin{corollary}
    Let
    \[
        D\eqdef
        \max
        \deg_{f^{k(c)}}(c),
    \]
    where the maximum is taken over all critical points $c\in f^{-1}(\Sigma)\setminus \Sigma$ and where $k(c)$ denotes the minimal integer such that $f^{k(c)}(c)$ is a periodic point. We have
    \[
        \alpha^+ \leq D\, \widetilde{\alpha}^+.
    \]
\end{corollary}

\appendix
\section{An alternative proof of the completeness of the spectrum. Specification Property}\label{s:alternative}

The purpose of this section is to give an alternative proof of Proposition~\ref{prop:monic}. We will obtain this proposition as an easy consequence of Lemma~\ref{lem:caf2}  below. We shall conclude the Appendix with a more precise version of this lemma, corresponding to Bowen's {\it periodic specification property}, \cite{Bow:71}.

\begin{lemma}\label{lem:caf2}
Given a neighborhood $V\supset \Sigma$, for every $\varepsilon>0$
there exists a periodic point $p\in J\setminus\Sigma$, so that
$\chi(p)\ge \chi^+(J\setminus V)-\varepsilon$.
\end{lemma}
\begin{proof}
We first collect some preliminary definitions and results.

Recall that $n\ge1$ is said to be a \emph{Pliss hyperbolic time} for $x$ with
exponent $\chi$ if
\begin{equation}\label{e.Pliss}
   \log\,\lvert (f^{n-m})'(f^m(x))\rvert\ge (n-m)\chi
   \quad\text{ for every }m=0, \ldots, n-1.
\end{equation}

Given $\varepsilon>0$, by the telescope lemma, see \cite[Lemma
9]{GelPrzRam:} or \cite{Prz:90}, there exist positive constants
$K_1$, $R_1$ so that for every $r\in(0,R_1)$,  every Pliss
hyperbolic time $n$ for a point $x$ with exponent $\chi>0$, and
every $m=0$, $\ldots$, $n-1$ we have
\begin{equation}\label{eq:shrink}
   \diam B_m\le r K_1 e^{-(n-m)(\chi-\varepsilon)},
\end{equation}
where
\begin{equation}\label{def.pull}
   B_m\eqdef f^{-(n-m)}_{f^m(x)}\big(B(f^n(x),r)\big).
\end{equation}
Now let $r\in(0,\min\{R_1,\dist(\Sigma,\partial V\})$.

Let us briefly write $\chi^+=\chi^+(J\setminus V)$. By definition of $\chi^+$ there exists $\widetilde N\ge1$ so that
for every $n\ge \widetilde N$ we have
\begin{equation}\label{e.defchi}
    \sup_{x\in J\setminus V}\frac{a(x,n)}{n}
   \le \chi^+ + \varepsilon , \quad
   \text{ where }a(x,n)\eqdef \log\,\lvert (f^n)'(x)\rvert.
\end{equation}

On the other hand, for every $n$ large enough we can choose a point $x=x(n)\in J\setminus V$ so that $a(x,n)>n(\chi^+-\varepsilon/2)$. Notice that we can assume that $n$ is a Pliss hyperbolic time for $x(n)$ with exponent $\chi^+-\varepsilon$ and satisfies $n\ge \widetilde N$.

 More precisely for the original $n$ let $n'\in\{1,\ldots, n\}$ be an integer such that at $m=n'$ the expression
 \[
    A(m)\eqdef
    a(x,m) - m (\chi^+-\varepsilon)
\]
attains its maximum. Clearly $n'$ is a
 Pliss hyperbolic time for $x$. Moreover, since $A(n)\ge n\varepsilon/2$, $A(0)=0$, and the function $\log\, \lvert f'\rvert$ is upper bounded, we obtain
 $n'\to\infty$ as $n\to\infty$. So we can replace $n$ by $n'$, thus assuming we have a sequence of pairs $\{(x_j,n_j)\}_j$ so that $x_j\in J\setminus V$, $n_j$ is a Pliss hyperbolic time for $x_j$ with exponent $\chi^+-\varepsilon$, $n_j\ge \widetilde N$, and $n_j\to\infty$ as $j\to\infty$. In the sequel we shall omit the index $j$.

We can assume that, possibly after slightly increasing
$\varepsilon$, additionally we have $f^n(x)\notin V$.
Indeed, let $m_-\in\{0, \ldots,n-1\}$ be the largest integer such
that $y= f^{m_-}(x)\notin V$. Since we assume that $n$ is  a Pliss
hyperbolic time for $x$ and since
 $y$ is close to a critical point, $\lvert f'(y)\rvert$ is small and hence the number $n-m_-$ must be large by~\eqref{e.Pliss}. In particular, $n-m_- \ge\card\Sigma$. Recall that $\Sigma$ contains periodic points together with their non-critical pre-images. Let the forward trajectory of $y$ follow a periodic trajectory of a point $q\in\Sigma$. Denote by $N_q$ the least period of $q$. Then, by~\eqref{e.Pliss} we have
\[
a(f^{n-N_q}(x),N_q)\ge N_q (\chi^+ -\varepsilon),
\]
which yields $\chi(q)\ge \chi^+ -2\varepsilon$ provided $V$ is
small enough, where the factor 2 takes in account the distortion
in a neighborhood of the trajectory of $q$. So in the case that
$n$ is a Pliss hyperbolic time and $f^n(x)\in V$ we can consider
the smallest integer $m_+>n$ for which $f^{m_+}(x)\notin V$. Then
there exists an integer $m_+'$ between $m_+$ and $m_+ + N_q$ which
is a Pliss hyperbolic time for $x$ with exponent
$\chi^+-3\varepsilon$. Note that $m_+$ and $m_+'$ exist, provided
$\chi^+-3\varepsilon>0$ and $V$ is small enough.

By~\eqref{e.defchi} for any point $y=f^m(x)$ with $m< n-\widetilde
N$ and $f(y)\notin V$ we have
\[
   a(f(y),n-m-1)\le (n-m-1)(\chi^+ +\varepsilon).
\]
Thus, together with~\eqref{e.Pliss} with $\chi=\chi^+
-\varepsilon$, for any such $y$ we conclude
\[\begin{split}
(n-m)(\chi^+ - \varepsilon)
&\le a(y,n-m) = \log \,\lvert f'(y)\rvert +  a(f(y),n-m-1)\\
&\le\log\, \lvert f'(y)\rvert + (n-m-1)(\chi^+ + \varepsilon)
\end{split}\]
and hence
\[
   \log\,\lvert f'(y)\rvert
   \ge (n-m)(\chi^+ -\varepsilon) - (n-m-1)(\chi^++\varepsilon)
   > -(n-m)2\varepsilon.
\]
Therefore, such $y$ must be in some distance to critical points
and satisfy
\begin{equation}\label{eq:dist}
   \dist(y,\Crit)
   \ge C\cdot \,e^{-(n-m)2\varepsilon},
\end{equation}
where $C$ is some positive constant.

As concluded before, $B\eqdef B(f^n(x),r)$ and $\Sigma$ are
disjoint. Now we pull back $B$ and show that for $r$ small enough
no pullback $B_m$ defined in~\eqref{def.pull} contains a critical
point.  To show this let us assume that the initially chosen $r$
also satisfies $r<C/K_1\exp(\chi^+-4\varepsilon)$ and that
$\chi^+-4\varepsilon>0$.

First, if $m$ satisfies $0\le m<n-\widetilde N$ we consider two cases:\\[0.1cm]
\textbf{1) $f^{m+1}(x)\notin V$:}
   Then~\eqref{eq:dist} and~\eqref{eq:shrink} for $\chi=\chi^+-\varepsilon$ imply $B_m \cap \Crit=\emptyset$.\\[0.1cm]
\textbf{2) $f^{m+1}(x)\in V$:}
   Then~\eqref{eq:shrink} implies that $B_m$ is very small. So, if there were a critical point $c\in B_m$ then we would have $f(c)\in\Sigma$. Since $\Sigma$ is forward invariant, this would imply $\Sigma\cap B(f^n(x),r)\ne\emptyset$ which is a contradiction. Hence $B_m\cap\Crit=\emptyset$. \\[0.1cm]
Second, in the remaining finite number of cases if $m$ satisfies
$n-\widetilde N\le m<n$ we can assure, possibly after decreasing
$r$, not depending of $(x,n)$ (possible since $n$ are Pliss
hyperbolic times with common $\chi$), that
$B_m\cap\Crit=\emptyset$.
Thus, we can conclude that none of the pullbacks $B_m$ captures a
critical point and therefore $f^{n-m}$ is univalent on $B_m$ for
every $m=0$, $\ldots$, $n-1$.

Given $r$, by Lemma~\ref{lem:caf1} there exists $\delta>0$ and
positive integers $N$, $M$, $i$, and $j$ with $N\le M$ and $0\le
i$, $j\le M$, and a point $z\in f^{-j}(f^i(x))$ such that $A\eqdef
f^{-N}(z)$ is $r/2$-dense in $J$ and satisfies
$\dist(f^k(A),\Crit)\ge\delta$ for every $k=0$, $\ldots$, $N+j-1$.

Now  we can choose $\delta'\in(0,\delta)$ independent of $x$ so
that for this point $z\in f^{-j}(f^i(x))$ and every $k=0$,
$\ldots$, $N+j-1$ we have
\[
   f^k\Big(f^{-N}\big(\Comp_zf^{-j}(B(f^i(x),\delta'))\big)\Big)\cap\Crit=\emptyset,
\]
yielding that $f^{N+j}$ is univalent on $B(w)$ defined by
\[
   B(w)\eqdef f^{-(N+j)}_w\big(B(f^i(x),\delta')\big)
\]
for every $w\in A$ as well as
\[
   \diam B(w) \le r/3.\]

   Thus, if $n$ is large enough so that
$rK_1\exp(-(n-i)(\chi^+-2\varepsilon))<\delta'$ then with $y\eqdef f^i(x)$ and using \eqref{eq:shrink}
with $\chi=\chi^+ -\varepsilon$ we obtain
\[
   f^{-(n-i)}_y\big(B(f^n(x),r)\big) \subset B(y,\delta').
\]
and hence for some choice of $w\in A$ (that will in general depend
on $f^n(x)$) we have
\[
   \widetilde B(w)\eqdef f^{-(n-i+N+j)}\big(B(f^n(x),r)\big)
   \subset B(f^n(x),5r/6)
\]
and $f^{n-i+N+j}$ is univalent on $\widetilde B(w)$. Hence the
latter set contains a periodic point  $p$. Using \eqref{e.Pliss}
and distortion estimates, it can be achieved that $\chi(p)\ge
\chi^+-3\varepsilon$, if $n$ is large enough. This proves the
lemma.
\end{proof}


\begin{remark*}{\rm
The idea of the proof of Lemma~\ref{lem:caf2} is taken from an unpublished note~\cite{Prz-letter}, where the supremum in~\eqref{e.sup} was taken over all $x\in J$ but  allowing the periodic point $p$ to belong to $\Sigma$.
The proof was simpler in that case as it did not rely on the telescope result~\eqref{eq:shrink}. The result in~\cite{Prz-letter} has been applied and referred to in~\cite{Gooran:98}.

Another proof of the weaker statement, that is, allowing $p\in \Sigma$, can be given by constructing a measure $\mu$ that is an accumulation of the sequence of measures $\frac 1n \sum_{k=0}^{n-1} \delta_{f^k(x)}$ as $n\to\infty$, where $\delta_y$ denotes the Dirac measure supported at $y$. Here $x$ and $n$ should be chosen to give an approximation of $\chi^+$ by $a(x,n)/n$.
One finds $p$ using Katok's method, see~\cite[Chapter 11.6]{PrzUrb:}.
This easy `ergodic' proof is in fact a part of the proof in Section~\ref{sec:final-3}.

Notice that the proof in Section~\ref{sec:final-3} does not yield the periodic specification property below because it bases on a specification of, maybe short,   special sub-blocks of a given piece of a trajectory.
}\end{remark*}

\begin{proof}[Proof of Proposition~\ref{prop:monic}]
By Lemma~\ref{lem:caf2} for every $x\notin\Sigma$ we have
\[
\overline\chi(x)\le \sup\{ \chi(p)\colon p\in J\setminus\Sigma\text{ periodic expanding }\}.
\]
By Corollary~\ref{def} this bound is less than or equal to $\widetilde\alpha^+$. This proves the proposition.
\end{proof}

\begin{proposition}\label{prop:specification}
    For every $V$ and $\e$ as in Lemma~\ref{lem:caf2} there exist an integer $N>0$ and $\e_1>0$ such that for every point $x\in J\setminus V$ and $n\ge N$ with
$f^n(x)\notin V$ satisfying $a(x,n)\ge \chi^+(J\setminus V)-\e_1$,
there exists an integer $m\in\{n(1-\e), \ldots, n\}$ and a
periodic point $p\in J$ of period at most $n+N$, such that
\begin{enumerate}
\item[1.] $\displaystyle\chi(p)\ge \chi^+(J\setminus V)-\e$\\[-0.3cm]
\item[2.] $\displaystyle {\rm{dist}}(f^j(x),f^j(p))\le \exp
(-(m-j)(\chi^+(J\setminus V)-\e))$ for all $j=4,5,...,m$.
\end{enumerate}
\end{proposition}

\begin{proof}
We just look more carefully at the proof of Lemma~\ref{lem:caf2}.
Consider an arbitrary small $t>0$. Assume that $x$ and $n$ satisfy
$a(x,n)\ge n(\chi^+ -t\e)$.

By definition of $\chi^+=\chi^+(J\setminus V)$ we can also assume, compare (\ref{e.defchi}), that
$a(x,m)\le m(\chi^+ +t\e)$ for all $m\in\{N(t\e),\ldots,n-1\}$, for some constant $N(t\e)$ depending only on $t\e$.

Then we find a number $n'\le n$ that is a Pliss hyperbolic time for $x$ with
 exponent $\chi^+-\varepsilon$, as in the proof of Lemma~\ref{lem:caf2}.  To estimate $n'$, notice that
\[
A(n)=a(x,n) - n (\chi^+-\varepsilon)\ge n\e (1-t),
\]
whereas
\[
A(m)=a(x,m) - m(\chi^+-\varepsilon)\le m(\chi^+ + t\e) - m(\chi^+ -\e)= m\e (1+t).
\]
Hence, for $m < n {1-t \over 1+t}$, we obtain $A(m)<A(n)$.
Notice also that for $m\le N(t\e)$ we have $A(m)<n\e(1-t) \le A(n)$ for $n$ large enough since
$\sup\, \lvert f'\rvert <\infty$.
In consequence, the positive integer $n'=m$
maximizing $A(m)$ is bigger or equal to $n {1-t \over 1+t}$.
Finally we choose $t$ so that  ${1-t \over 1+t}>1-\e$.

Next in the proof of Lemma~\ref{lem:caf2} we have increased $n'$ to achieve $f^{n'}(x)\notin V$.
Here the increase is not beyond $n$ since we have already assumed $f^n(x)\notin V$.

The rest of the proof is the same. The mysterious indices $j=4,5,\ldots$  in
the assertion comes from Lemma \ref{lem:at most four} implying
that at most fourth iteration of any point has a backward branch
omitting critical points, hence $i\le 4$.
\end{proof}

\section*{Acknowledgement}

KG has been supported by the Alexander von Humboldt Foundation Germany, CNPq and FAPERJ, Brazil. The research of FP and MR were supported by the EU FP6 Marie Curie programmes SPADE2 and CODY and by the Polish MNiSW Grant NN201 0222 33 `Chaos, fraktale i dynamika konforemna'. JR-L was partially supported by FONDECYT N 1100922

\bibliographystyle{amsplain}

\end{document}